\documentclass[11pt]{article}
\usepackage{amssymb,amsmath,amsthm,algorithm,algorithmic}
\usepackage{bbm}
\usepackage{graphicx}
\usepackage{amsmath}
\usepackage{amsfonts}
\usepackage{amssymb}
\usepackage{enumerate}
\usepackage{lscape}
\usepackage{longtable}
\usepackage{rotating}
\usepackage{color}
\usepackage{url}
\usepackage{subfigure}
\usepackage{rotating}

\renewcommand{\d}{{\delta}}
\newcommand{\beqn}{\begin{equation}}
\newcommand{\eeqn}{\end{equation}}
\newcommand{\E}{{\mathbb{E}}}
\newcommand{\one}{{\mathbbm{1}}}
\newtheorem{theorem}{Theorem}[section]

\newtheorem{corollary}{Corollary}[section]

\newtheorem{definition}{Definition}[section]

\newtheorem{lemma}{Lemma}[section]

\newtheorem{remark}{Remark}[section]

\newtheorem{assumption}{Assumption}[section]

\begin{document}

\title{Convergence Rate Analysis of a Stochastic  Trust Region Method via Supermartingales}
%
\author{Jose Blanchet\thanks{Department of Industrial Engineering and Operations Research, Columbia University
500 West 120th street, New York, NY 10027, USA. {\tt jose.blanchet@columbia.edu.}}
\and   Coralia Cartis\thanks{Mathematical Institute, University of Oxford, Radcliffe Observatory Quarter, Woodstock Road,
Oxford, OX2 6GG, United Kingdom. {\tt cartis@maths.ox.ac.uk}. This work was partially supported by the Oxford University EPSRC Platform Grant
 EP/I01893X/1.}
 \and  Matt Menickelly\thanks{Mathematics and Computer Science Division, Argonne National Laboratory, Lemont, IL 60439, USA. {\tt mmenickelly@anl.gov}.
 The work of this author was partially supported by NSF Grant DMS 13-19356.}
\and Katya Scheinberg\thanks{Department of Industrial and Systems Engineering, Lehigh University,
Harold S. Mohler Laboratory, 200 West Packer Avenue, Bethlehem, PA 18015-1582, USA.
{\tt katyas@lehigh.edu}. The work of this author was  partially supported by NSF Grants, DMS 13-19356, CCF 13-20137 and CCF 16-18717. It was partially performed
 while the author was visiting the University of Oxford, sponsored by the Visiting Professorship Grant VP1-2016-053 from the Leverhulme Trust. } 
}
\date{\today}

\maketitle

\begin{abstract}
We propose a novel framework for analyzing convergence rates of stochastic optimization algorithms with adaptive step sizes. This framework is based on analyzing properties of an underlying generic stochastic process, in particular by deriving a bound on the expected stopping time of this process. We utilize this framework to analyze the bounds on expected global convergence rates of a stochastic variant of a traditional trust region method, introduced in  \cite{ChenMenickellyScheinberg2014}. While traditional trust region methods rely on exact 
computations of the gradient, Hessian and values of the objective function, this method assumes that these values are available up to some dynamically adjusted accuracy. Moreover, this accuracy is assumed to hold only with some  sufficiently large, but fixed, probability, without any additional restrictions on the variance of the errors. This setting applies, for example,  to standard stochastic optimization and machine learning formulations.  Improving upon the analysis in  \cite{ChenMenickellyScheinberg2014}, we show that the stochastic process defined by the algorithm satisfies the assumptions of our proposed general framework, with the stopping time defined  as reaching accuracy  $\|\nabla f(x)\|\leq \epsilon$.  The resulting bound for this stopping time is $O(\epsilon^{-2})$, under the assumption of sufficiently accurate stochastic gradient, and is the first global complexity bound for a stochastic trust-region method.  Finally, we apply the same framework to derive second order complexity bound under some additional assumptions. 
\end{abstract}

\section{Introduction}
In this paper we aim to solve   a stochastic unconstrained, possibly nonconvex, optimization problem 
 \begin{eqnarray}\label{eq:object}
\min\limits_{x\in \mathbb{R}^n} \,  f(x)
\end{eqnarray}
where  $f(x)$ is a function which is  assumed to be smooth and bounded from below, and whose value can only be computed with some noise. 
Let $\tilde{f}(x,\xi)$ be the noisy computable version of $f$, where the noise $\xi$ is a random variable. One of the most common settings is
\[
f(x)=  \mathbb{E}_\xi[ \tilde{f}(x,\xi) ] .
\]

Stochastic optimization methods, in particular stochastic gradient descent (SGD), have recently become the focus of much research in optimization, especially in applications to machine learning domains. This  is because in machine learning the objective function of the optimization problem is typically a sum of a (possibly) very large number of terms, each term being the loss function evaluated  using  one data example. This objective function can also be viewed as an expected loss, in which case it cannot be accurately computed, but can only be evaluated approximately, given a subset of data samples. During the last decade significant theoretical and algorithmic advances 
were developed for convex optimization  problems, such as logistic regression and support vector machines. However, with the recent practical success of deep neural networks and other nonlinear, nonconvex ML models, the focus has shifted to the analysis and development of methods for nonconvex optimization problems. 
While SGD remains the method of choice in the nonconvex setting for ML applications, theoretical results are weaker than those in the convex case. In particular,
little has been achieved in terms of convergence rates. A notable paper \cite{GhadimiLan} is the first to provide convergence rates guarantee of a sort for a randomized stochastic gradient method in nonconvex setting. This method, however, utilizes a carefully chosen step size and a randomized stopping scheme, which are quite different from what is used in practice. 

As an alternative to the basic SGD several {\em variance reducing} stochastic methods have been proposed recently, such as SAGA \cite{SAGA}, SVRG  \cite{SVRG} and SARAH \cite{nguyen2017sarah}. They enjoy much stronger convergence rates than SGD, and have  been extended to nonconvex problems \cite{nonconvexSVRG,nguyen2017sarah_nonconvex}. However,
these methods specifically exploit the structure of the ML problems, where the objective function is a sum over a deterministic (if large) set of data. 
SVRG requires the full gradient of the objective function to be computed on some (but not all) of the iterations. Hence, SVRG, essentially is a hybrid between  SGD and the full gradient method applied to a finite sum. The overall convergence rate per number of data accesses for SVRG seems better than those for the other two  methods. From the practical perspective, however,  SGD has low per-iteration complexity and high number of iterations and thus is not  effective in a distributed setting, while each iteration of a 
full gradient method can be efficiently distributed, reducing the overall wall-clock time.   SVRG (as well as SARAH) being  a hybrid does not easily fit with either setting because it alternates between cheap stochastic gradient computation, which have to be sequential and expensive full gradient computations, which can be distributed. In other words, the superior theoretical computational complexity of the SVRG does not necessarily reflect its practical performance. Moreover, the assumption that the data set is fixed (deterministic) contradicts the ultimate goal of learning, which is to obtain a solution with good generalization performance.  The method we describe in this paper  is applicable in the purely  stochastic setting (without assuming that there is a fixed finite set of data) and as our theory shows, it relies on variance reduction that can simply be achieved by choosing 
adaptive sample sizes that tend to grow  as the algorithm progresses to optimality. Such adaptive schemes have been proposed in the literature primarily for gradient descent methods and in a convex setting \cite{RByrd_etal_2012,Schmidt}.  

With the rise of interest in nonconvex optimization, the ML community started to consider a classical alternative to gradient descent/line search methods - trust region methods \cite{TRBook, Lin07trustregion, Dauphinetal_2014}. Their usefulness is largely dictated by their ability to utilize negative curvature in Hessian approximations and hence, potentially, escape the neighborhoods of saddle points \cite{Dauphinetal_2014}, which can significantly slow down  or even trap a line search method. It is argued, that 
while saddle points are undesirable, the local minima are typically sufficient for the purposes of training nonconvex ML models, such as deep neural networks. There has been a number of recent works that propose trust-region methods that use stochastic gradient and Hessian estimates \cite{GrattonEtAl2017, Mahoney}, but they all assume  that the objective function is deterministic.  One of the original stochastic trust region  methods for this stochastic optimization setting has been proposed in \cite{ChangHongWan} and a more sophisticated adaptive method has been recently introduced  in \cite{2015sarhasetal}.
For both methods the convergence is achieved by repeatedly sampling the function values (and gradients, when applicable) so that eventually the  estimates
 become asymptotically error-free with probability 1. No convergence rates have been derived for these algorithms, because the  progress happens in the assymptotics.  Fully stochastic versions of trust region methods with adaptive sampling, such as may be used  in ML context, have not yet been explored to our knowledge. In addition, our analysis applies to the setting where the objective functions and gradient estimates may be biased.

The method, which we analyze in this paper and which we will refer to as STORM (Stochastic trust region method with Random Models), has been introduced in  \cite{ChenMenickellyScheinberg2014}, where the authors prove the almost-sure convergence 
to a first order stationary point.  It is a stochastic, variance reducing  trust region method  which is essentially a minor modification of a classical trust region framework. A very similar method has been analyzed in \cite{larson2013stochastic}, under more restrictive conditions on $\tilde f(x,\xi)$. We believe that our convergence rate analysis framework can be easily applied to that method as well, but we choose to focus on STORM. 

STORM uses adaptive trust region radii and is close to what is known to be efficient in practice, hence here we focus  on the theoretical analysis of this method in the first and second order settings. We recover the convergence rates whose dependence on $\epsilon$ is the same as of those for deterministic trust region method. Since the method is stochastic our convergence rates are derived in the form of the bound on the expected number of iterations the algorithms takes until achieving the $\epsilon$-accuracy. In contrast, convergence rate result for SGD in \cite{GhadimiLan}, for example, only bounds the expected  sum of the norms of all the gradients up to iteration $T$, as a function of $T$. Other weaker types of convergence rates are established in \cite{Mahoney} and \cite{Tripuraneni2017}. In \cite{Mahoney} a trust region and a cubic regularization 
methods based on sampled Hessian are considered. The number of samples is selected in such a way that the error in the Hessian approximation is smaller than $\epsilon$ with large probability $p$. Then the deterministic convergence rate can be established under the assumption that the condition on the Hessian approximation  holds at each iteration until $\epsilon$-accuracy is reached. Hence, the  established  bound on the number of iterations $T$ 
holds with probability $p^T$ and with probability $1-p^T$ no bound is known. 
The same type of complexity result is derived in  \cite{Tripuraneni2017} for a cubic regularization method, where  gradients and Hessians are also sampled at a rate dictated by $\epsilon$ and the resulting bound holds only with some probability. 

Algorithms in \cite{Mahoney} have some similarities with algorithms analyzed in 
\cite{CartisScheinberg2014} and \cite{GrattonEtAl2017}.   In  \cite{CartisScheinberg2014} and \cite{GrattonEtAl2017}, the global rates of convergence of a trust region, a line search and an adaptive cubic regularization methods are analyzed under the assumption that first and second order information is inexact, but sufficiently accurate with some probability, however, the analysis in all of these papers relies heavily on the assumption that function values are computed accurately, in particular that no increasing steps are allowed.  This implies that the results in  \cite{CartisScheinberg2014} and \cite{GrattonEtAl2017} cannot be applied 
in a stochastic setting. \cite{Tripuraneni2017}, on the other hand, does not explicitly use function values, because it does not utilize adaptive step sizes. 
This paper can be seen as an extension of  \cite{CartisScheinberg2014} and \cite{GrattonEtAl2017} to the case of stochastic functions.


Unlike what is done in most of the literature on stochastic methods,  we do not make the assumption that  
the function, gradient or Hessian estimates are unbiased. 
Instead it is assumed  that at each iteration, the function values $f(x)$, the gradient $\nabla f(x)$ and possibly the Hessian   $\nabla^2 f(x)$  can be approximated up to sufficient accuracy with a fixed, but sufficiently high probability $p$, conditioned on the past. This assumption, which we will make formal later, is very general and does not explicitly specify how such approximations can be obtained. In case when unbiased estimators are available, one can utilize sampling techniques described, for example in 
\cite{Mahoney, ChenMenickellyScheinberg2014}. On the other hand,  in \cite{ChenMenickellyScheinberg2014} that are examples of $\tilde f(x,\xi)$ which is a biased estimator of $f(x)$, which is arbitrarily erroneous with some small fixed probability, and yet, the required approximations can be constructed and the trust region method converges. Note that while our condition on the approximations hold only with probability $p$, we provide complexity result in expectation, thus accounting for "failed" approximations.

The goal of our paper is twofold: First, we introduce a novel framework for bounding expected complexity of a stochastic optimization method. This framework is based on defining a renewal-reward  process associated with the algorithm as well as  its stopping time, which is the time when the algorithm reaches desired accuracy. Then, under certain assumptions, we derive a bound on the expected stopping time. This framework, in principal, can be used for analysis of convergence rates of a variety of algorithms - for instance it applies to all algorithms in  \cite{CartisScheinberg2014} and \cite{GrattonEtAl2017}.  
In recent work \cite{PaquetteScheinberg2018} it has been applied to analyze a stochastic line-search method. 
In this paper, specifically, we  use this general framework to derive a bound on the  convergence rate of the  STORM algorithm defined  in \cite{ChenMenickellyScheinberg2014}, by proving that  these assumptions are satisfied by this algorithm. 
In particular, we  show that the expected number of iterations required to achieve
$\|\nabla f(x)\|\leq \epsilon$ is bounded by $O(\epsilon^{-2}/(2p-1))$, which is an improvement on the result in \cite{GhadimiLan} and a similar one to those in \cite{nonconvexSVRG,nguyen2017sarah_nonconvex}, in terms of dependence on $\epsilon$, but such that, in principal, it never requires computation of the true gradient. The result is a natural extension of the standard, 
best-known worst-case complexity of any first order method for nonconvex optimization \cite{Nesterov}. 
In this paper we also make a significant improvement upon the results in \cite{ChenMenickellyScheinberg2014} by relaxing a very restrictive condition on the size of the steps taken by the algorithm. 
By applying the general analytic framework again, we  also provide a second order complexity analysis. We show that a second order
STORM variant takes an expected number of iterations that is at most
 $O(\epsilon^{-3}/(2p-1))$ to ensure
$\max\{\|\nabla f(x)\|,-\lambda_{\min}(\nabla^2 f(x))\}\leq \epsilon$; this result requires slightly stronger assumptions on the
function estimates but provides generalization of results in \cite{Mahoney,GrattonEtAl2017} to the stochastic case. 

Our main complexity results does not yet provide a termination criterion that would guarantee that $\|f(\bar x)\|\leq \epsilon$, where $\bar x$ is the last iterate. However, the analysis provides a foundation for establishing such a criterion. In particular, while in this paper we simply bound the expected complexity, bounding the tail of the complexity distribution will follow from the analysis here.  

In the next section we present and analyze our generic framework, while the  STORM algorithm is analyzed in Section \ref{sec:firstorder}. 

The rest of the paper is organized as follows: we begin by introducing the stochastic framework and deriving the bound on its expected stopping  time in Section \ref{sec:walds}. In Section \ref{sec:firstorder} we provide the first order complexity analysis of the STORM algorithm by showing that it fits into the framework introduced  in Section \ref{sec:walds}. The second order complexity analysis follows in  Section \ref{section:second-order}.
\paragraph{Notation} 
Throughout the paper we use $\|\cdot\|$ to denote the Euclidian norm. Some of the bounding constants that are used are denoted by $\kappa$ with a subscript that is meant to indicated the objects that the constant bounds.  In particular we use the following constants. 

\begin{equation*}
    \begin{aligned}
\kappa_{ef} & \mbox{\quad``error in the function value",}\\
\kappa_{eg} &  \mbox{\quad``error in the gradient",}\\
\kappa_{fcd} & \mbox{\quad``fraction of Cauchy decrease",}\\
\kappa_{bhm}&\mbox{\quad``bound on the Hessian of the models",}\\
     \end{aligned}
\end{equation*}
We will also use $I\left( A\right )$ to denote the indicator of a random event $A$ occurring. 
\section{A Renewal-Reward Martingale Process}

\label{sec:walds}

In this section we consider a general random process and a stopping time $T$,
which posses certain properties. We analyze the behavior of this random
process and derive a bound on the expected stopping time. These results will
be used later in the paper in the specific setting of convergence 
of a stochastic trust region method to first order stationary points. We argue that the framework presented in 
this section can be used for convergence analysis of a variety of stochastic algorithms.
We start by defining a stopping time of a discrete time stochastic process.

\begin{definition}
\label{def:stoptime} Given a stochastic process $\{X_{k}\}=\{X_{k}:k\geq0\}$,
we say that $T$ is a stopping time with respect to $\{X_{k}\}$ if for each
$m\geq0$ the occurrance of the event $\{T=m\}$ is determined by observing
$X_{1},\dots,X_{m}$. That is, $\{T=m\}\in\sigma\left(  X_{0},...,X_{m}\right)
$, the $\sigma$-field generated by $X_{1},...,X_{m}$, for each $m\geq0$.
\end{definition}

Now, let $\{\left(  \Phi_{k},\Delta_{k}\right)  \}$ be a random process such
that $\Phi_{k}\in\lbrack0,\infty)$ and $\Delta_{k}\in\lbrack0,\infty)$ for
$k\geq0$. Let $V_{k+1}=\Phi_{k+1}-\Phi_{k}$ for $k\geq0$. We also assume the
existence of a sequence $\left\{  W_{k}\right\}  _{k=1}^{\infty}$, defined on
the same probability space as $\{\left(  \Phi_{k},\Delta_{k}\right)  \}$, we
introduce $W_{0}=1$ and let $\mathcal{F}_{k}$ denote the $\sigma$-algebra
generated by $\{\left(  \Phi_{0},\Delta_{0},W_{0}\right)  ,\cdots,\left(
\Phi_{k},\Delta_{k},W_{k}\right)  \}$. We assume that $\left\{  W_{k}\right\}
_{k=1}^{\infty}$ satisfies
\begin{equation}%
\begin{array}
[c]{rcl}%
P(W_{k+1}=1|\mathcal{F}_{k}) & = & p,\\
P(W_{k+1}=-1|\mathcal{F}_{k}) & = & 1-p.
\end{array}
\label{w_process}%
\end{equation}
Note that under the assumption (\ref{w_process}) the $W_{k}$'s are independent
and also independent of the sequence $\left\{  \left(  \Phi_{k},\Delta
_{k}\right)  \right\}  $. 

Let $\left\{  T_{\epsilon}\right\}  _{\epsilon>0}$ be a family of stopping
times with respect to $\left\{  \mathcal{F}_{k}\right\}  _{k\geq0}$,
parametrized by some quantity $\epsilon>0$. The following assumptions will be
imposed on $\{\left(  \Phi_{k},\Delta_{k}\right)  \}$ and $T_{\epsilon}$.

\begin{assumption}
\label{ass:stoch-proc} \ \newline

\begin{enumerate}
\item[(i)] There exists  constants $\lambda\in\left(  0,\infty\right)  $ and
$\Delta_{\max}=\Delta_{0}e^{  \lambda j_{\max}}  $ (for some
$j_{\max}\in\mathbb{Z}$) 
such that $\Delta_{k}\leq\Delta_{max}$ for all $k$.

\item[(ii)] There exists a constant 
$\Delta_{\epsilon}=\Delta_{0}e^{  \lambda j_{\epsilon}}  $ (for some
$j_{\epsilon}\in\mathbb{Z}$, $j_{\epsilon}\leq 0$) such that, the following holds for each $k\geq
0$,
\begin{equation}
I\left(  T_{\epsilon}>k\right) \Delta_{k+1} \geq I\left(  T_{\epsilon
}>k\right)  \min(\Delta_{k}e^{\lambda W_{k+1}},\Delta_{\epsilon}),
\label{deltaprocess}%
\end{equation}
where $W_{k+1}$ satisfies \eqref{w_process} with $p>\frac{1}{2}$.

\item[(iii)] There exists a nondecreasing function $h(\cdot):[0,\infty
)\rightarrow(0,\infty)$ and a constant $\Theta>0$ such that
\begin{equation}
{\mathbb{E}}(V_{k+1}|\mathcal{F}_{k})I\left(  T_{\epsilon}>k\right)
\leq-\Theta h(\Delta_{k})I\left(  T_{\epsilon}>k\right)  \label{vmartingale}%
\end{equation}
or, equivalently,
\begin{equation}
{\mathbb{E}}(\Phi_{k+1}|\mathcal{F}_{k})I\left(  T_{\epsilon}>k\right)
\leq\Phi_{k}I\left(  T_{\epsilon}>k\right)  -\Theta h(\Delta_{k})I\left(
T_{\epsilon}>k\right)  . \label{phisubmartingale}%
\end{equation}

\end{enumerate}
\end{assumption}

In summary, Assumption \ref{ass:stoch-proc} states that the nonnegative
stochastic process $\Phi_{k}$ gets reduced by at least $\Theta h(\Delta_{k})$
at each step, as long as $T_{\epsilon}>k$. Also, $\Delta_{k}$ tends to
increase whenever it is smaller than some threshold $\Delta_{\epsilon}$. Our
goal is to bound ${\mathbb{E}}(T_{\epsilon})$ in terms of $h(\Delta_{\epsilon
})$. What we will show in this section is that, on average, $\Delta_{k}%
\geq\Delta_{\epsilon}$ frequently, and hence, ${\mathbb{E}}(\Phi_{k+1}-\Phi_{k})$ can be
bounded by a negative fixed value (dependent on $\epsilon$), sufficiently
frequently, which will allow us to apply Wald's identity (stated below) and derive the bound on
${\mathbb{E}}(T_{\epsilon})$. In order to formalize this, we introduce a
renewal process, where renewals occur at times when $\Delta_{k}\geq
\Delta_{\epsilon}$ and we consider the sum of $V_{j}$'s obtained between two renewals.

In order to define this renewal process we first introduce an auxiliary
process. Define $\left\{  Z_{k}\right\}  _{k=0}^{\infty}$ as follows. First,
let $Z_{0}=j_{\epsilon}$ and set
\[
Z_{k+1}=\min(Z_{k}+W_{k+1},j_{\epsilon}).
\]
Note that the process $\left\{  Z_{k}\right\}  _{k=0}^{\infty}$ is a
birth-death process on the set $\left\{  k:k\leq j_{\epsilon}\right\}  $.
Then, define the renewal process, $A_{0}=0$ and $A_{n}=\inf\{m>A_{n-1}%
:Z_{m}=j_{\epsilon}\}$. By Assumption (\ref{deltaprocess}) and using a simple
inductive argument for the second inequality below we have that
\[
I\left(  T_{\epsilon}>k\right)\Delta_{k+1} \geq I\left(  T_{\epsilon
}>k\right)  \min(\Delta_{k}e^{\lambda W_{k+1}},\Delta_{\epsilon})\geq I\left(
T_{\epsilon}>k\right)  \Delta_{0}\exp\left(  \lambda Z_{k+1}\right)  .
\]

In other words, on $T_{\epsilon}>k$, the process $A_{n}$ only counts the
iterations for which $\Delta_{k}$ has value at least $\Delta_{\epsilon}$. The
interarrival times of this renewal process are defined for all $k\geq1$ by
\[
\tau_{n}=A_{n}-A_{n-1},
\]
As a final piece of notation, we define the counting process
\[
N(k)=\max\{n:A_{n}\leq k\},
\]
which is the number of renewals that occur before time $k$.

First, we have a lemma which relies on the simple structure of the process
$\{W_{k}\}$ to bound ${\mathbb{E}}[\tau_{n}]$.

\begin{lemma}
\label{tau_1_bound} Let $\tau_{n}$ be defined as above. Then, for all $n$
\[
{\mathbb{E}}[\tau_{n}]=p+\left(  1+\frac{1}{2p-1}\right)  \left(  1-p\right)
=p/(2p-1).
\]

\end{lemma}

\begin{proof}
Define the process $\bar{Z}_{k+1}=\bar{Z}_{k}+W_{k+1}$, which is a simple
random walk. Suppose that $\bar{Z}_{0}=-1$ and define $\bar{\tau}=\inf
\{n\geq0:\bar{Z}_{n}=0\}$ it is well known (in fact, this follows by Wald's
identity) that%
\[
{\mathbb{E}}\left(  \bar{\tau}\right)  =\frac{1}{2p-1}.
\]
On the other hand, by first step analysis (i.e. conditioning on $W_{1}$) we
have that
\[
{\mathbb{E}}[\tau_{1}]=1\cdot P\left(  W_{1}=1\right)  +(1+{\mathbb{E}}%
[\bar{\tau}])P\left(  W_{1}=-1\right)  .
\]
The above identity follows because the distribution of $\tau_{1}$
conditioned on $Z_{1}=j_{\epsilon}-1$ is the same as the distribution of
$\bar{\tau}$. So, we conclude that%
\[
{\mathbb{E}}[\tau_{1}]=p+\left(  1+\frac{1}{2p-1}\right)  \left(  1-p\right)
,
\]
the last equality follows by simplifying the expression above.
\end{proof}

We now bound the number of renewals that can occur before the time
$T_{\epsilon}$.

\begin{lemma}
\label{n_t_eps_bound}
\[
{\mathbb{E}}(N(T_{\epsilon}-1)+1)\leq\frac{\Phi_{0}}{\Theta h(\Delta
_{\epsilon})}.
\]

\end{lemma}

\begin{proof}
For ease of notation, let $k\wedge T_{\epsilon}=\min\{k,T_{\epsilon}\}$.
Consider the stochastic process defined via $R_{0}=\Phi_{0}$ and
\[
R_{k}=\Phi_{k\wedge T_{\epsilon}}+\Theta\sum_{j=0}^{\left(  k\wedge
T_{\epsilon}\right)  -1}h(\Delta_{j}),
\]
for $k\geq1$, where $\Theta$ is defined in \eqref{phisubmartingale}. Observe
that $R_{k}$ is a non-negative supermartingale with respect to $\left\{
\mathcal{F}_{k}\right\}  $, to see this we first write%
\[
{\mathbb{E}}[R_{k+1}|\mathcal{F}_{k}]={\mathbb{E}}[R_{k+1}I\left(
T_{\epsilon}>k\right)  |\mathcal{F}_{k}]+{\mathbb{E}}[R_{k+1}I\left(
T_{\epsilon}\leq k\right)  |\mathcal{F}_{k}].
\]
Then,
\begin{align}
{\mathbb{E}}[R_{k+1}I\left(  T_{\epsilon}\leq k\right)  |\mathcal{F}_{k}]  &
={\mathbb{E}}[R_{k+1}I\left(  T_{\epsilon}\leq k\right)  |\mathcal{F}%
_{k}]\nonumber\\
&  ={\mathbb{E}}\left[  \left(  \Phi_{T_{\epsilon}}+\Theta\sum_{j=0}%
^{T_{\epsilon}-1}h(\Delta_{j})\right)  I\left(  T_{\epsilon}\leq k\right)
|\mathcal{F}_{k}\right] \nonumber\\
&  =\Phi_{T_{\epsilon}}I\left(  T_{\epsilon}\leq k\right)  +\Theta\sum
_{j=0}^{T_{\epsilon}-1}h(\Delta_{j})I\left(  T_{\epsilon}\leq k\right)  ,
\label{B1}%
\end{align}
where the last equality follows because $T_{\epsilon}$ is a stopping time and
therefore the random variable the expectation is $\mathcal{F}_{k}$-measurable.

On the other hand, since $\left\{  T_{\epsilon}\geq k+1\right\}  =\left\{
T_{\epsilon}>k\right\}  =\left\{  T_{\epsilon}\leq k\right\}  ^{c}%
\in\mathcal{F}_{k}$ we conclude, using \eqref{phisubmartingale}, that
\begin{align}
&  {\mathbb{E}}[R_{k+1}I\left(  T_{\epsilon}>k\right)  |\mathcal{F}%
_{k}]  ={\mathbb{E}}[R_{k+1}|\mathcal{F}_{k}]I\left(  T_{\epsilon}>k\right)
\nonumber\\
&  ={\mathbb{E}}[\Phi_{k+1}|\mathcal{F}_{k}]I\left(  T_{\epsilon}>k\right)
+{\mathbb{E}}\left[  \Theta\sum_{j=0}^{k}h(\Delta_{j})|\mathcal{F}_{k}\right]
I\left(  T_{\epsilon}>k\right) \nonumber\\
&  \leq\left(  \Phi_{k}-\Theta h(\Delta_{k})+\Theta\sum_{j=0}^{k}h(\Delta
_{j})\right)  I\left(  T_{\epsilon}>k\right) \nonumber\\
&  =\left(  \Phi_{k}+\Theta\sum_{j=0}^{k-1}h(\Delta_{j})\right)  I\left(
T_{\epsilon}>k\right)  . \label{B2}%
\end{align}
Combining (\ref{B1}) and (\ref{B2}) we conclude that%
\[
{\mathbb{E}}[R_{k+1}|\mathcal{F}_{k}]\leq R_{k}%
\]
as claimed. We then conclude, since $\Phi_{k}\geq0$ for each $k\geq0$, that%
\[
\Theta{\mathbb{E}}\left(  \sum_{j=0}^{\left(  k\wedge T_{\epsilon}\right)
-1}h(\Delta_{j})\right)  ={\mathbb{E}}[R_{k}]\leq{\mathbb{E}}[R_{0}]=\Phi
_{0}.
\]
Now, since $h(\cdot)\geq0$, observe that
\[
0\leq\sum_{j=0}^{\left(  k\wedge T_{\epsilon}\right)  -1}h(\Delta_{j}%
)\nearrow\sum_{j=0}^{T_{\epsilon}-1}h(\Delta_{j})
\]
as $k\rightarrow\infty$, note that this conclusion holds also on the event
$\{T_{\epsilon}=\infty\}$. Therefore, by the Monotone Convergence Theorem
\begin{equation}
\Theta{\mathbb{E}}\left(  \sum_{j=0}^{T_{\epsilon}-1}h(\Delta_{j})\right)
=\lim_{k\rightarrow\infty}\Theta{\mathbb{E}}\left(  \sum_{j=0}^{\left(
k\wedge T_{\epsilon}\right)  -1}h(\Delta_{j})\right)  \leq{\mathbb{E}}%
[R_{0}]=\Phi_{0}. \label{thm1}%
\end{equation}
Now, by the definition of the counting process $N(\cdot)$, since the renewal
times $A_{n}$ when $\Delta_{A_{n}}\geq\Delta_{\epsilon}$, are a subset of the
iterations $0,1,\dots,T_{\epsilon}$, and since $h(\cdot)$ is nondecreasing, we
have
\[
\Theta\sum_{j=0}^{T_{\epsilon}-1}h(\Delta_{j})\geq\Theta\sum_{j=0}%
^{T_{\epsilon}-1}h(\Delta_{j})I\left(  j\in\left\{  A_{i}\right\}
_{i=1}^{\infty}\right)  =\Theta\left(  N(T_{\epsilon}-1)+1\right)
h(\Delta_{\epsilon}),
\]
the term +1 being added to $N(T_{\epsilon}-1)$ comes from the fact that
$A_{0}=0$. Inserting this in \eqref{thm1},
\[
{\mathbb{E}}(\left(  N(T_{\epsilon}-1)+1\right)  )\leq\frac{\Phi_{0}}{\Theta
h(\Delta_{\epsilon})},
\]
which concludes the proof.
\end{proof}

We now state and prove a well known theorem on  expected stopping time, known as 
the Wald's Identity (non-negative increments case) (e.g., see Theorem 2.2.4 in \cite{Alsmeyer10}).  The reason we provide a proof here, is that 
in the literature this result is typically shown under the assumption that the stopping time is finite a.s. 
Dropping this condition is particularly important in our framework, as this condition is equivalent to a 
convergence result for the optimization algorithm which generates the stochastic process. 
It is convenient and useful not to have to prove the convergence result before establishing the convergence rates bounds, since the
convergence immediately follows from these bounds. 

\begin{theorem} {\textbf{Wald's Identity (non-negative increments).}}
Suppose that $\left\{ Y_{i}\right\} _{i=1}^{n}$ is a sequence of independent
random variables such that $Y_{i}\in \lbrack 0,\infty ]$ with probability
one. Define $E\left( Y_{i}\right) =\mu _{i}\in \lbrack 0,\infty ]$ and let $%
N\in \lbrack 0,\infty ]$ be a stopping time with respect to the filtration
generated by the $Y_{n}$'s. Define $S_{n}=Y_{1}+...+Y_{n}$, $S_{0}=0$, $%
s_{n}=\mu _{1}+...+\mu _{n}$ and $s_{0}=0$. Then 
\[
E\left( S_{N}\right) =E\left( s_{N}\right) . 
\]
\end{theorem}

\begin{proof}
Let $m>0$ be an arbitrary integer and define $Y_{i}\left( m\right) =\min
\left( Y_{i},m\right) $, $N_{m}=\min \left( N,m\right) $, $\mu _{i}\left(
m\right) =E\left( Y_{i}\left( m\right) \right) $, $S_{n}\left( m\right)
=Y_{1}\left( m\right) +...+Y_{n}\left( m\right) $ and $s_{n}\left( m\right)
=\mu _{1}\left( m\right) +...+\mu _{n}\left( m\right) $. Note that all of
these quantities are non-negative and non-decreasing in $m$. By the optional
sampling theorem applied to the martingale $M_{n}=S_{n}\left( m\right)
-s_{n}\left( m\right) $, we have that%
\[
E\left( S_{N_{m}}\left( m\right) \right) =E\left( \mu _{1}\left( m\right)
+...+\mu _{N_{m}}\left( m\right) \right) .
\]%
Because of monotonicity, 
\[
S_{N_{m}}\left( m\right) \nearrow S_{N}
\]%
as $m\rightarrow \infty $. For the case $N=\infty $, we interpret $%
S_{N}=\sup_{n\geq 0}\sup_{m}S_{n}\left( m\right) $. Similarly, 
\[
s_{N_{m}}\left( m\right) \nearrow s_{N},
\]%
as $m\rightarrow \infty $. By the monotone convergence theorem we then
conclude that%
\[
E\left( S_{N}\right) =E\left( s_{N}\right) .
\]
\end{proof}

\bigskip

\begin{remark} If $\mu _{i}=\mu $, then $E\left( S_{N}\right) =\mu \cdot E\left(
N\right) $. If $\mu =0$, then $X_{i}=0$ almost surely and $S_{N}=0$.
Therefore, if $\mu =0$, we interpret $\mu \cdot E\left( N\right) =0$, even
if $E\left( N\right) =\infty $. This interpretation is consistent with the
case in which $N=0$ almost surely as well, in this case $0=\mu \cdot E\left(
N\right) =E\left( S_{N}\right) $, even $\mu =\infty $.
\end{remark}

%

We now apply this
theorem to $S_n=A_{n}=\sum_{i=0}^{n}\tau_{i}$ and obtain the main result of this
section, which will be used in the following sections to establish the main
complexity result.

\begin{theorem}
Let Assumption \ref{ass:stoch-proc} hold. Then \label{t_eps_bound}
\[
{\mathbb{E}}[T_{\epsilon}]\leq\frac{p}{2p-1}\cdot\frac{\Phi_{0}}{\Theta
h(\Delta_{\epsilon})}+1.
\]

\end{theorem}

\begin{proof}
Define $\mathcal{G}_{n}=\mathcal{F}_{A_{n}}$, that is,
\[
\mathcal{G}_{n}=\{A\in\sigma\left(  \cup_{m=0}^{\infty}\mathcal{F}_{m}\right)
:A\cap\{A_{n}\leq k\}\in\mathcal{F}_{k}\text{ for all }k\}.
\]
Note that $A_{n}$ is a stopping time with respect to $\left\{  \mathcal{F}%
_{n}\right\}  _{n\geq0}$, so $\mathcal{G}_{n}$ is well defined. We claim that
$N\left(  T_{\epsilon}-1\right)  +1$ is a stopping time with respect to
$\left\{  \mathcal{G}_{n}\right\}  _{n\geq0}$. To see this, note, since
$N\left(  k\right)  \leq k$
\begin{align*}
&  \left\{  N\left(  T_{\epsilon}-1\right)  +1\leq n\right\} \\
&  =\cup_{k=0}^{n-1}\left\{  N\left(  k\right)  \leq n-1,T_{\epsilon
}-1=k\right\} \\
&  =\cup_{k=0}^{n-1}\left\{  N\left(  k\right)  +1\leq n,T_{\epsilon
}=k+1\right\}  \subseteq\mathcal{F}_{A_{n}},
\end{align*}
where the last inclusion follows because $N\left(  k\right)  +1$ is a stopping
time with respect to $\left\{  \mathcal{F}_{A_{n}}\right\}  _{n\geq0}$ and
because $A_{n}\geq n$, so $\mathcal{F}_{n}\subseteq\mathcal{F}_{A_{n}}$, which
implies that $T_{\epsilon}$ is also stopping time with respect to $\left\{
\mathcal{G}_{n}\right\}  _{n\geq0}$.

Now, because of the independence assumption implied by (\ref{w_process}) we
have that%
\[
{\mathbb{E}}[\tau_{n+1}|\mathcal{G}_{n}]={\mathbb{E}}[\tau_{n+1}]=\frac{p}{2p-1}.
\]
Recalling that $A_{N(T_{\epsilon}-1)+1}= \sum_{k=1}^{N(T_{\epsilon}-1)+1}\tau_k$,
we can invoke Wald's identity to conclude that
\[
{\mathbb{E}}[A_{N(T_{\epsilon}-1)+1}]=\frac{p}{2p-1}{\mathbb{E}}%
[N(T_{\epsilon}-1)+1].
\]
Since $A_{N(T_{\epsilon}-1)+1}\geq T_{\epsilon}-1$, we have by Lemmas
\ref{tau_1_bound} and \ref{n_t_eps_bound}
\[
{\mathbb{E}}[T_{\epsilon}-1]\leq{\mathbb{E}}[\tau_{1}]{\mathbb{E}%
}[N(T_{\epsilon}-1)+1]\leq\frac{p}{2p-1}\left(  \frac{\Phi_{0}}{\Theta
h(\Delta_{\epsilon})}\right)  .
\]
The statement of the theorem follows from the last inequality.
\end{proof}

\section{The first order STORM algorithm}\label{sec:firstorder}
We now  state  and analyze a stochastic  trust region (TR) algorithm (Algorithm 1) which is essentially very similar to its deterministic counterpart \cite{TRBook}.
This method uses the inexact (noisy) information about $f$ and its derivatives, just as the deterministic method uses the exact information. 
This algorithm, as stated, and the assumptions on its steps that we will impose below aim at convergence to a first order stationary point. In this section we will analyze the global rate of convergence of this algorithm to such a point
(while in Section \ref{section:second-order}, we extend Algorithm \ref{algo:stodfosimple} to calculate second order critical
points).
\begin{algorithm}
\caption{{\sc Stochastic DFO with Random Models,  \cite{ChenMenickellyScheinberg2014}}}
\label{algo:stodfosimple}
 \begin{algorithmic}[1]
    \STATE (Initialization):  Choose  constants $\gamma>1$, $\eta_1\in(0,1)$, $\eta_2 > 0$. Choose an initial point $x^0$ and an initial trust-region radius $\delta_0>0$ and the maximum radius $\delta_{\max}=\gamma^{j_{\max}}\delta_0$ for some $j_{\max}\geq 0$. Set $k \gets 0$.
\STATE\label{step.model} (Model construction): Build a (random) model $m_k(x_k+s)=f_k+g_k^\top s+\frac{1}{2}s^\top H_k s$ that approximates $f(x)$ in the ball $B(x_k, \delta_k)$ with $s=x-x_k$.
\STATE (Step calculation) Compute $s_k = \arg \underset{s: \| s \|\leq \delta_k}{\min}  m_k(s) $ (approximately) so that it satisfies condition \eqref{eqn:CS}. 
\STATE (Estimates calculation) Obtain estimates $f_k^0$ and $f_k^s $ of $f(x_k)$ and $f(x_k+s_k)$, respectively.
\STATE (Acceptance of the trial point): Compute $\rho_k = \dfrac{f_k^0-f_k^s}{m_k(x_k)-m_k(x_k+s_k)}.$\\ If $\rho_k\geq \eta_1$ and $ \| g_k\| \geq \eta_2 \delta_k $, set $x_{k+1} =x_k+s_k$; otherwise, set $x_{k+1}=x_k$.
\STATE (Trust-region  radius update): If $\rho_k\geq \eta_1$ and $ \| g_k\| \geq\eta_2 \delta_k $, set $\delta_{k+1} =\min\{  \gamma \delta_k,\delta_{\max}\}$; otherwise $\delta_{k+1}=\gamma^{-1} \delta_k$; $k \gets k+1$ and go to step \ref{step.model}.
  \end{algorithmic}
\end{algorithm}

For every $k$, the  step $s_k$ is computed so that the well-known {\em Cauchy decrease} condition is satisfied,
\begin{eqnarray}
m_k(x_k)-m_k(x_k+s_k) \ge \dfrac{\kappa_{fcd}}{2} \| g_k \| \min \left\{ \dfrac{\|g_k\|}{\|H_k\|},\d_k  \right\}  \label{eqn:CS}
\end{eqnarray}
for some constant $\kappa_{fcd}\in(0,1].$ This condition is standard for the TR methods, easy to enforce in practice and is discussed in detail in the literature \cite{TRBook,NoceWrig06}. Iterations on which $x_{k+1} =x_k+s_k$ occurs are called {\em successful}. 

Algorithm \ref{algo:stodfosimple} generates a random process. The source of randomness are the random models and random estimates constructed on each iteration, based on some random information obtained from the  stochastic function $f(x,\varepsilon)$. $M_k$  will denote a random model in the $k$-th iteration, while we will use the notation $m_k=M_k(\omega)$ for its realizations. As a consequence of using  random models, the iterates $X_k$, the trust-region radii $\Delta_k$ and the steps $S_k$  are also random quantities, and so $x_k=X_k(\omega)$, $\d_k = \Delta_k(\omega)$, $s_k=S_k(\omega)$ will denote their respective realizations.  Similarly,
 let random quantities  $\{F_k^0,F_k^s\}$ denote the estimates of $f(X_k)$ and $f(X_k+S_k)$, with their realizations denoted by $f_k^0=F_k^0(\omega)$ and $f_k^s=F_k^s(\omega)$.  
 In other words, Algorithm \ref{algo:stodfosimple} results in a stochastic process $\{M_k,X_k,S_k, \Delta_k, F_k^0, F_k^s\}$. 
 Our goal is to show that under certain conditions on the sequences $\{M_k\}$
and $\{F_k\}=\{(F_k^0, F_k^s)\}$ the resulting stochastic process has desirable convergence rate. In particular, we will 
assume that  models $M_k$
and estimates  $F_k^0, F_k^s$ are sufficiently accurate with sufficiently high probability, conditioned on the past. 

The key to the analysis lies in the assumption that the accuracy improves in coordination with the perceived progress of the algorithm. 
The main challenge of the analysis lies in the fact that, while in the deterministic case the function $f(x)$ never increases from one iteration to another, this can easily happen in the stochastic case. The analysis is based on properties of supermartingales where the increments of a supermartingale depend on the function change between iterates (which as we will show, {\em tend} to decrease). To make the analysis simpler we need a technical assumption that these increments are bounded from above. 
Hence, overall we make the following assumptions on $f$:
\begin{assumption}\label{ass:F}
We assume that all iterates $x_k$ generated by Algorithm \ref{algo:stodfosimple} 
the  gradient $\nabla f$ is $L$-Lipschitz continuous 
and
\[  f(x)\geq 0 
\]
\end{assumption}

The assumptions of Lipschitz continuity of $\nabla f$ and boundedness of $f$ from below are standard. Here for simplicity and w.l.o.g. we assume that the lower bound on $f$ is nonnegative. 

\subsection{Assumptions on the first order STORM algorithm}
Let $\mathcal{F}_{k-1}^{M\cdot F}$ denote the $\sigma$-algebra generated by $M_0,\cdots, M_{k-1}$ and $F_0,\cdots,F_{k-1}$ and let $\mathcal{F}_{k-{1}/{2}}^{M\cdot F}$ denote the $\sigma$-algebra generated by $M_0,\cdots, M_{k}$ and $F_0,\cdots,F_{k-1}$. 

\begin{definition}
1) A function $m_k$ is a \emph{$\kappa$-fully linear model of $f$ on $B(x_k,\d_k)$} provided, for $ \kappa = (\kappa_{ef}, \kappa_{eg})$ and  $\forall y \in B(x_k,\delta_k)$,
\begin{eqnarray}\label{eq:fl-model}
\| \nabla f(x_k) - g_k  \| & \le & \kappa_{eg} \d_k, \\
| f(y)-m_k(y)| &\le & \kappa_{ef} \d_k^2.\nonumber
\end{eqnarray}

2) The estimates $f_k^0$ and $f_k^s$ are said to be $\epsilon_F$-accurate estimates of $f(x_k)$ and $f(x_k+s_k)$, respectively,  for a given $\delta_k$ if
\begin{equation}\label{eq:estimates}
 |f_k^0 - f(x_k) |  \le \epsilon_F \d_k^2 \mbox{ and } |f_k^s - f(x_k+s_k) | \le \epsilon_F \d_k^2.
 \end{equation}

\end{definition} 

\begin{definition}\label{def:pfl-models}
A sequence of random models $\{ M_k \}$ is said to be $\alpha$-probabilistically $\kappa$-fully linear with respect to the corresponding sequence $\{ B(X_k,\Delta_k)\}$ if the events 
\begin{equation}\label{eq:Ik}
I_k = \one\{ M_k \mbox{ is a } \kappa \mbox{-fully linear model of } f \mbox{ on }B(X_k,\Delta_k)  \}
\end{equation}
satisfy the condition
\[ P (I_k=1 | \mathcal{F}_{k-1}^{M\cdot F})\ge \alpha.\]
\label{probabilisticmodel}
\end{definition}
\begin{definition}\label{def:pa-estimates} A sequence of random estimates $\{F_k\}$ is said to be $\beta$-probabilistically $\epsilon_F$-accurate with respect to the corresponding sequence $\{X_k,\Delta_k,S_k\}$ if the events
\begin{equation}\label{eq:Jk}
 J_k =\one \{F_k^0,F_k^s \mbox{ are } \epsilon_F\mbox{-accurate estimates of }f(x_k) \mbox{ and }f(x_k+s_k), \mbox{ respectively, \ for\ } \Delta_k  \}  
 \end{equation}
satisfy the condition 
\[P( J_k=1|\mathcal{F}_{k-1/2}^{M\cdot F}  )\ge \beta,\]
where $\epsilon_F$ is a fixed constant. 
\label{asmpesti}
\end{definition}

 Next is the key assumption on the nature of the stochastic (and deterministic) information used by our algorithm.
\begin{assumption}\label{ass:models_estim} The following hold for the quantities used in Algorithm \ref{algo:stodfosimple}

\begin{itemize}
\item[(a)] The model Hessians satisfy $\|H_k\|_2\leq \kappa_{bhm}$ for some $\kappa_{bhm}\geq 1$, for all $k$, deterministically. 


\item[(b)] The sequence of random models $M_k$, generated by Algorithm \ref{algo:stodfosimple}, is $\alpha$-probabilistically $\kappa$-fully linear, for some $\kappa= (\kappa_{ef}, \kappa_{eg})$ and for a sufficiently large $\alpha\in (0,1)$. 

\item[(c)] The sequence of random estimates $\{F_k\}$ generated by Algorithm \ref{algo:stodfosimple} is 
 $\beta$-probabilistically $\epsilon_F$-accurate 
for   $\epsilon_F\leq \kappa_{ef}$ and 
$\epsilon_F< \frac{1}{4}\eta_1\eta_2\kappa_{fcd}\min\left\{\frac{\eta_2}{\kappa_{bhm}},1\right\}$, 
and for a sufficiently large $\beta\in (0,1)$.

\end{itemize}
\end{assumption}
\begin{remark}\label{rem:eta2}
In  \cite{ChenMenickellyScheinberg2014} 
 the analysis of the algorithm requires an additional assumption that $\eta_2\geq \kappa_{ef}$ and for simplicity it is further  assumed  that $\eta_2\geq \kappa_{bhm}$.  This assumption is undesirable since it restricts the size of the steps that can be taken by the trust region algorithm.  In this paper we manage to improve the analysis and drop this assumption, hence allowing $\eta_2$ to be set to a small value. 
 Note that small values of $\eta_2$ imply small $\epsilon_F$ because of  Assumption \ref{ass:models_estim}(c), so there is a potential trade-off in choosing $\eta_2$. On the other hand, this relationship indicates that when $\epsilon_F=0$, that is when there is no error in the function estimates, then 
 $\eta_2$ can be arbitrarily small. 
 \end{remark}

Under Assumption \ref{ass:models_estim}, $P\{I_k J_k=1|\mathcal{F}_{k-1}^{M\cdot F}\}\geq \alpha\beta$ and $P\{I_k+J_k=0|\mathcal{F}_{k-1}^{M\cdot F}\}\leq (1-\alpha)(1-\beta)$. At iteration $k$, if $I_kJ_k=1$ then the behavior of the algorithm reduces to that of an (inexact) deterministic algorithm; while if $I_k+J_k=0$, then not only may the algorithm produce a bad step (that is a step which increases the objective function), but it also may accept this bad step by mistaking it for an improving step (that is a step that decreases the function value). In the cases when only one of $I_k=0$ and $J_k=0$ holds, then either the model is good but the estimates are faulty, or the estimates are good and the model is faulty. In this case an improving step is still possible, but a bad step is not. In the worst case, no step is taken and the trust region radius is reduced. The main idea of our framework is to choose probabilities of $I_kJ_k=1$ and $I_k+J_k=0$ occurring according to the possible corresponding decrease and increase in $f(x)$, so that in expectation, $f(x)$ is sufficiently reduced.

\subsection{Useful existing results}
Algorithm \ref{algo:stodfosimple} is analyzed in \cite{ChenMenickellyScheinberg2014} and  the following almost-sure stationarity result is shown:
 there exists a selection of $\alpha$ and $\beta$
 such that under Assumption \ref{ass:models_estim} with additional requirement that $\eta_2\geq \kappa_{ef}$, the sequence of random iterates generated by Algorithm \ref{algo:stodfosimple},
 $\{X_k\}$, almost surely satisfies $ \underset{k\to \infty}{\lim} \| \nabla f(X_k)  \|=0$.
The important observation is that $\alpha$ and $\beta$ do not have to increase as the algorithm progresses. Hence with the same, constant but small enough, probabilities our models and estimates can be arbitrarily erroneous.

%


Our primary goal in this paper is to bound the expected number of steps that the algorithm takes until $\| \nabla f(X_k)  \|\leq \epsilon$  occurs and the secondary goal is to relax the assumption $\eta_2\geq \kappa_{ef}$. 
We will modify the analysis  that led to the above stationarity result in  \cite{ChenMenickellyScheinberg2014}.
 First, we state (without proof) several auxiliary lemmas from \cite{ChenMenickellyScheinberg2014}.
 \begin{lemma}\label{lemma.delta.1}
{\em [Good model $\Rightarrow$ function reduction in $\|g_k\|$]} Suppose that a model $m_k$ is a $(\kappa_{ef},\kappa_{eg})$-fully linear model of $f$ on $B(x_k,\d_k)$. If
$$\d_k\le\min\left\{ \frac{1}{\kappa_{bhm}},\frac{\kappa_{fcd}}{8\kappa_{ef}}  \right\} \|g_k\|,$$
then the trial step $s_k$ leads to an improvement in $f(x_k+s_k)$ such that
\begin{eqnarray}f(x_k+s_k)-f(x_k)\le - \frac{\kappa_{fcd}}{4} \|g_k\|\d_k .
\end{eqnarray}
\end{lemma}

\begin{lemma}\label{lemma.delta.2} {\em [Good model $\Rightarrow$ function reduction in $\|\nabla f(x_k)\|$]} Under Assumption \ref{ass:models_estim}(a), suppose that a model is $(\kappa_{ef},\kappa_{eg})$-fully linear on $B(x_k,\d_k)$. If
\begin{eqnarray}\label{condition.lemma.delta.2}
\d_k\le \min \left\{ \frac{1}{\kappa_{bhm}+\kappa_{eg}}, \frac{1}{\frac{8\kappa_{ef}}{\kappa_{fcd}}+\kappa_{eg}}\right\} \|\nabla f(x_k)\|,
\end{eqnarray}
then the trial step $s_k$ leads to an improvement in $f(x_k+s_k)$ such that
\begin{eqnarray}f(x_k+s_k)-f(x_k)\le - C_1 \|\nabla f(x_k)\|\d_k , \label{eqn.true.decrease}
\end{eqnarray}
for any $C_1\leq \frac{\kappa_{fcd}}{4}\cdot \max\left\{ \frac{\kappa_{bhm}}{\kappa_{bhm}+\kappa_{eg}},\frac{8\kappa_{ef}}{8\kappa_{ef}+\kappa_{fcd}\kappa_{eg}}\right\}.$
\end{lemma}

\begin{lemma}\label{lemma.delta.3} { \em [Good model $+$ good estimates $\Rightarrow$ successful step]}  Under Assumption \ref{ass:models_estim}(a), suppose that $m_k$ is $(\kappa_{ef},\kappa_{eg})$-fully linear on $B(x_k,\d_k)$ and the estimates $\{ f_k^0,f_k^s \}$ are $\epsilon_F$-accurate with $\epsilon_F\le \kappa_{ef}$. If 
\begin{eqnarray}\label{condition.lemma.delta.3}
 \d_k \le \min \left\{ \dfrac{1}{\kappa_{bhm}} , \frac{1}{\eta_2}, \dfrac{ \kappa_{fcd}  (1-\eta_1)}{ 8\kappa_{ef} }   \right\} \| g_k \|, 
 \end{eqnarray}
then the $k$-th iteration is successful.
\end{lemma}

\begin{lemma}\label{lemma.delta.4} { \em [ Good estimates $+$ successful step $\Rightarrow$ function reduction in $\delta_k^2$]} 
Under Assumption \ref{ass:models_estim}(a), suppose that the estimates $\{f_k^0,f_k^s  \}$ are $\epsilon_F$-accurate with $\epsilon_F< \frac{1}{4}\eta_1\eta_2 \kappa_{fcd}\min \left\{ \frac{\eta_2}{\kappa_{bhm}},1 \right\}$. If a trial step $s_k$ is accepted (a successful iteration occurs), then the improvement in $f$ is bounded below as follows
\begin{eqnarray}\label{eqn.lemma.4}
f(x_{k+1})-f(x_k)\le -C_2\delta_k^2,\end{eqnarray}
where 
\begin{equation}\label{eq:C2}
C_2 =\frac{1}{2}\eta_1\eta_2 \kappa_{fcd}\min \left\{ \frac{\eta_2}{\kappa_{bhm}},1 \right\}-2\epsilon_F>0.
\end{equation}
\end{lemma}

\paragraph{Choosing constants}\label{page:constants} 
We now explain briefly the role of the constants  $\eta_2,\epsilon_F,\alpha,$ and $\beta$ and their expected magnitude. First, we note that constants $\kappa_{ef},\kappa_{eg}$, and $\kappa_{bhm}$ can be chosen arbitrarily large, but ideally should be chosen as small as possible while guaranteeing Assumption \ref{ass:models_estim}. Let us assume that $\kappa_{ef},\kappa_{eg}$ and $\kappa_{bhm}$ can be all chosen as $\Theta(L)$\footnote{Note that it is possible to have 
$\kappa_{ef}$ and $\kappa_{eg}$ of different magnitudes, namely when $\kappa_{eg}$ is small as we have sufficiently accurate gradients, but $\kappa_{ef}$ remains as $\Theta(L)$. Our analysis and results apply then as well.}, where $L$ is the Lipschitz constant of $\nabla f(x)$ in $\cal{X}$, even though it may not be explicitly known (see \cite{DFObook} for construction of fully-linear models in the case of unavailable derivative estimates). Once these constants are chosen, $\epsilon_F$ is chosen so that it satisfies the conditions in Assumption \ref{ass:models_estim}(c).  Note that if $\eta_2$ is chosen  to equal $L$, this means that Algorithm \ref{algo:stodfosimple} only takes steps when $\delta_k \leq \frac{ \|g_k\|}{L}$, which is similar to constraining  step size by $\frac{1}{L}$ in gradient descent. In this case, $ \epsilon_F$ can be  chosen 
relatively large and thus  the estimates need to be slightly more accurate than the models, but the order of required accuracy is similar. 
On the other hand, since in trust region methods, step sizes are meant to be chosen adaptively, it  is desirable to allow larger steps, which can be done by setting $\eta_2$ to be small. This, in turn,  implies that $ \epsilon_F$ has to be chosen small  and hence the estimates will have to be a lot more accurate than the models. Another trade-off when choosing a small value for $\eta_2$ will become apparent in our main complexity results, as we will see that the expected improvement per iteration may depend on $\eta_2$. But reasonable values for $\eta_2$ allow the removal of this dependency.

To simplify expressions for various constants we will assume that $\eta_1=0.1$, $\gamma=2$ and $\kappa_{fcd}=0.5$ which are typical values for these constants. We also assume, w.l.o.g., that $\kappa_{bhm}\leq 12\kappa_{ef}$ and $\eta_2\leq \kappa_{eg}$. 
To simplify expressions further we will consider $\kappa_{ef}=\kappa_{eg}$. It is clear that if $\kappa_{ef}$ or $\kappa_{eg}$ happen to be smaller, somewhat better bounds than the ones  we derive here will result, because the models give tighter approximations of the true function. We are interested in deriving bounds for the case when $\kappa_{ef}$ or $\kappa_{eg}$ may be large. The analysis can be performed for any other values of the above constants, hence the choice here is done merely for convenience and simplicity.

The conditions on $\alpha$ and $\beta$ under the above choice of constants will be shown in our results below.  
 
 \subsection{Defining and analyzing the process $\{\Phi_k, \Delta_k\}$.}
We consider a random process $\{\Phi_k, \Delta_k\}$ derived from the  process  generated by Algorithm \ref{algo:stodfosimple}, with $\Delta_k$ - the trust region radius and 
\begin{equation}\label{eq:Phidef}
\Phi_k = \nu f(X_k) + (1-\nu)\Delta_k^2
\end{equation}
where $\nu\in (0,1)$ is a deterministic, large enough constant, which will be defined later. 
Clearly $\Phi_k\geq 0$. 
Recall the notation $\phi_k$ for realizations of $\Phi_k$. Here we defined $\mathcal{F}_{k}$ as $\mathcal{F}_{k}^{M\cdot F}$.

Define a random time 
\begin{equation}\label{eq:Tdef}
T_\epsilon = \inf\{k\geq 0 : \|\nabla f(X_k)\|\leq\epsilon\}.
\end{equation} 
It is easy to see that $T_\epsilon$ is a stopping time for the stochastic process defined by Algorithm \ref{algo:stodfosimple} and hence for $\{\Phi_k, \Delta_k\}$. 

As stated our goal is to bound the expected stopping time $\mathbb{E}(T_\epsilon)$. We will do so by showing that Assumption \ref{ass:stoch-proc} is satisfied for 
$\{\Phi_k, \Delta_k\}$ and applying results of Section \ref{sec:walds}.


Let us  show that Assumptions \ref{ass:stoch-proc}(i)-(ii) hold with the following $\Delta_{\epsilon}$

\begin{equation}\label{eq:zeta}
\Delta_\epsilon = \displaystyle\frac{\epsilon}{\zeta}, \quad {\rm for \  }\zeta\geq\kappa_{eg}+\max\{\eta_2, \kappa_{bhm}, \frac{8\kappa_{ef}}{\kappa_{fcd}(1-\eta_1)}\}.
\end{equation}
Note that with  our choice of algorithmic parameters the above  is satisfied by  $\zeta=20\kappa_{eg}$. 

For simplicity of the presentation and without loss of generality, we assume that  $\Delta_\epsilon=\gamma^i\delta_0$, for some integer 
$i\leq 0$. If not, we can always choose $\zeta$  within a factor of ${\gamma}$ of its lower bound in \eqref{eq:zeta}. 
It follows that for any $k$, $\Delta_k=\gamma^{i_k}\Delta_\epsilon$, for some integer $i_k$. Choosing $\lambda$
in Assumption \ref{ass:stoch-proc}(i)-(ii) so that $e^\lambda=\gamma$ Assumption \ref{ass:stoch-proc}(i) follows immediately from the definition of  $\{\Phi_k, \Delta_k\}$,  and the  choice of  $ \delta_{max}$ imposed by  Algorithm  \ref{algo:stodfosimple} and for Assumption \ref{ass:stoch-proc}(ii)  we now only need to show that the dynamics  \eqref{deltaprocess} hold for $\Delta_k$.

\begin{lemma}\label{lem:Deltak}
Let Assumptions \ref{ass:F} and  \ref{ass:models_estim} hold. Let $\alpha$ and $\beta$ be such that $\alpha\beta\geq 1/2$, then Assumption \ref{ass:stoch-proc}(ii)
is satisfied for $W_k=2(I_kJ_k-\frac{1}{2})$, $\lambda=\log(\gamma)$ and $p=\alpha\beta$. 
\end{lemma}
\begin{proof}
Clearly inequality \eqref{deltaprocess} holds when $I(T_\epsilon > k)=0$. 
We will  show that conditioned on $T_\epsilon > k$ (i.e. $I(T_\epsilon > k)=1$)  we have
\begin{equation}
\label{Deltakdynamics}
\Delta_{k+1}  \geq  \min\{\Delta_\epsilon, \min\{\Delta_{\max},\gamma\Delta_k\}I_kJ_k + \gamma^{-1}\Delta_k(1-I_kJ_k)\}.
\end{equation}
First we note that for each realization when $\delta_k>\Delta_\epsilon$, we have $\delta_k\geq \gamma\Delta_\epsilon$  and hence $\delta_{k+1}  \geq  \Delta_\epsilon$. 
Now, assume that $\delta_k\leq\Delta_\epsilon$, then, because $T_\epsilon > k$,  we have $\|\nabla f(x_k)\|>\epsilon$ and hence, from the definition of $\zeta$,  
we know that
\[
 \|\nabla f(x_k)\|\ge \left(\kappa_{eg}+ \max\left\{  \eta_2,  \kappa_{bhm}, \frac{8\kappa_{ef}}{\kappa_{fcd}(1-\eta_1)}\right\}\right)\delta_k.
 \]

Assume that $I_k=1$ and $J_k=1$, i.e., both the  model and the  estimates are good on iteration $k$.  
Since the model $m_k$ is $\kappa$-fully linear and 
 $$\|g_k\|\ge \|\nabla f(x_k)\|-\kappa_{eg}\d_k\ge (\zeta-\kappa_{eg})\d_k\ge \max \left\{\eta_2, \kappa_{bhm}, \frac{8\kappa_{ef}}{\kappa_{fcd}(1-\eta_1)}  \right\}\d_k,$$
 and the estimates $\{ f_k^0,f_k^s \}$ are $\epsilon_F$-accurate, with $\epsilon_F\le \kappa_{ef}$, condition (\ref{condition.lemma.delta.3}) in Lemma \ref{lemma.delta.3} holds. Hence, iteration $k$ is successful, i.e. $x_{k+1}=x_k+s_k$ and $\d_{k+1}=\max\{\delta_{max},\gamma\d_k\}$.
 If $I_kJ_k=0$, then $\d_{k+1}\geq\gamma^{-1}\d_k$ simply by 
 the dynamics of Algorithm  \ref{algo:stodfosimple}. 
 
Finally,  observing that $P\{I_kJ_k=1|\mathcal{F}_{k-1}^{M\cdot F}\}\geq p=\alpha\beta$ we conclude that  \eqref{Deltakdynamics} implies Assumption \ref{ass:stoch-proc}(ii). 
\end{proof}


We now show that Assumption \ref{ass:stoch-proc}(iii) holds, which is the key theorem in this section and is similar to Theorem 4.11 in    \cite{ChenMenickellyScheinberg2014}, while dropping the restrictive conditions on $\eta_2$
and simplifying the proof. We will omit the parts of the proof that are identical to those of Theorem 4.11 in    \cite{ChenMenickellyScheinberg2014}. 

\begin{theorem}\label{thm:main}
 There exist probabilities $\alpha$ and $\beta$  such that under Assumptions \ref{ass:F} and  \ref{ass:models_estim} 
 there exists a constant $\Theta>0$ 
such that, conditioned on $T_\epsilon > k$
\begin{equation}
\label{case1dynamics}
1(T_\epsilon > k)\E[\Phi_{k+1}-\Phi_{k}|\mathcal{F}_{k-1}^{M\cdot F} ]  \leq -1(T_\epsilon > k)\Theta \Delta_k^2. 
\end{equation} 
Moreover,  under the particular choice of constants described in the last section, let $\alpha$ and $\beta$ satisfy
\[
\frac{(\alpha\beta-\frac{1}{2})}{(1-\alpha)(1-\beta)}\ge 10 +\frac{30L}{40\kappa_{eg}},
\]
and 
$$
\beta\geq 
\frac{\kappa_{eg} +0.064 L +4\cdot 10^{-4}{\eta_2}}{ \kappa_{eg} + 0.064 L +4.5\cdot 10^{-4}\eta_2}. 
$$
Then\footnote{Note that $\beta>\frac{1}{2}$ and so $\Theta= \frac{1}{1800}\kappa_{eg}^{-1}$, independently of $\eta_2$, provided $\eta_2\geq 2\kappa_{eg}^{-1}$; the latter implies that  small values are allowed for $\eta_2$ as $\kappa_{eg}$ values of interest are large.}, $\Theta= \frac{1}{1800}\min \left\{\eta_2\beta,\kappa_{eg}^{-1}\right\}$.
\end{theorem}

\begin{proof} 
Since \eqref{case1dynamics} holds trivially if $T_\epsilon \leq k$, we assume henceforth in this proof that $\nabla f(X_k)>\epsilon$. 
We will consider two possible cases:  $\| \nabla f(x_k)  \| \ge \zeta \delta_k $ and $\| \nabla f(x_k)  \| < \zeta \delta_k$.
We will show that \eqref{case1dynamics}
holds in both cases and hence it holds on every iteration.
Let  $\nu\in (0,1)$ be such that
\begin{eqnarray}\label{eq:nu}
\frac{\nu}{1-\nu}&>& \max \left\{ \frac{4\gamma^2}{\zeta C_1} , \frac{4\gamma^2}{ \eta_1\eta_2\kappa_{fcd}}, \frac{\gamma^2}{\kappa_{ef}} \right\},
\end{eqnarray}
with $C_1$ defined as in Lemma \ref{lemma.delta.2}. 

Let $x_k$, $\delta_k$, $s_k$, $g_k$, and $\phi_k$ denote realizations of random quantities $X_k$, $\Delta_k$, $S_k$, $G_k$, and  $\Phi_k$, respectively.
Let us consider some realization of Algorithm  \ref{algo:stodfosimple}. Note that on all successful iterations, $x_{k+1}=x_k+s_k$ and $\delta_{k+1}=\min\{\gamma \delta_k, \delta_{max}\}$ with $\gamma>1$, hence
\begin{eqnarray} 
\phi_{k+1}-\phi_k\leq\nu (f(x_{k+1})-f(x_k))+(1-\nu)(\gamma^2-1)\delta_k^2.\label{eqn.suc.phi}
\end{eqnarray} 
On  all unsuccessful iterations, $x_{k+1}=x_k$ and $\delta_{k+1}=\frac{1}{\gamma} \delta_k$, i.e. 
\begin{eqnarray}
\phi_{k+1}-\phi_k=(1-\nu)(\frac{1}{\gamma^2}-1)\delta_k^2\equiv b_1<0.\label{eqn.unsuc.phi}
\end{eqnarray}

\paragraph{Case 1: $\| \nabla f(x_k)  \| \ge \zeta \delta_k $
with  $\zeta$ satisfying \eqref{eq:zeta}. }

Let $\alpha$ and $\beta$ satisfy
\begin{eqnarray}\label{eq:alpha.beta.condition.1}
\frac{(\alpha\beta-\frac{1}{2})}{(1-\alpha)(1-\beta)}&\ge& \frac{ C_3 }{C_1},
\end{eqnarray}
with $C_1$ defined in Lemma \ref{lemma.delta.2} and $C_3=1+\frac{3L}{2\zeta}$.

\begin{itemize}
\item[a.] $I_k=1$ and $J_k=1$, i.e., both the  model and the  estimates are good on iteration $k$.  
The proof is almost identical to that in 
Theorem 4.11   \cite{ChenMenickellyScheinberg2014}, but with small modification due to the different definition of $\zeta$ because we no longer assume that $\eta_2\geq  \kappa_{bhm}$. 


%
By observing that Lemma \ref{lemma.delta.2} and Lemma \ref{lemma.delta.3} hold we can derive 
\begin{eqnarray}
\phi_{k+1}-\phi_k\le -\nu C_1\| \nabla f(x_k)\|\delta_k +(1-\nu)(\gamma^2-1)\delta_k^2\equiv b_2,\label{eqn.suc.phi.truedecrease1}
\end{eqnarray}
for $\nu\in(0,1)$ satisfying \eqref{eq:nu}. 


\item[b.] $I_k=1$ and  $J_k=0$, i.e., we have a good model and bad estimates  on iteration  $k$. 
The proof is identical to that in 
Theorem 4.11   \cite{ChenMenickellyScheinberg2014} where it is shown that \eqref{eqn.unsuc.phi} holds.

\item[c.] $I_k=0$  and  $J_k=1$, i.e., we have a bad model and good estimates  on iteration  $k$. 
Again \eqref{eqn.unsuc.phi} holds, as is shown in
Theorem 4.11 \cite{ChenMenickellyScheinberg2014}. 
%
%

\item[d.] $I_k=0$ and $J_k=0$, i.e., both the  model and the  estimates are bad on iteration $k$.
The proof of  
Theorem 4.11    \cite{ChenMenickellyScheinberg2014} applies, where it is shown 
that 
\begin{eqnarray}
\phi_{k+1}-\phi_k\le \nu C_3 \|\nabla f(x_k)\| \delta_k+(1-\nu)(\gamma^2-1)\delta_k^2\equiv b_3.\label{eqn.suc.phi.falseincrease}
\end{eqnarray}
 holds with  $C_3=1+\frac{3L}{2\zeta}$. 

%
\end{itemize}
\vskip0.5cm
Next, following the proof of Case 1 of Theorem 4.11 in    \cite{ChenMenickellyScheinberg2014} we combine the four outcomes to obtain 
that under condition \eqref{eq:alpha.beta.condition.1}, we have
\begin{equation}\label{eq:expectPhi1}
\E[\Phi_{k+1}-\Phi_{k}|\mathcal{F}_{k-1}^{M\cdot F},  \{  \|\nabla f(X_k)\|\geq \zeta \Delta_k \}] \le  -\frac{1}{4} C_1 \nu \|\nabla f(X_k)\|\Delta_k\leq  -\frac{1}{4} \frac{C_1 \nu}{\zeta} \Delta_k^2.\nonumber
\end{equation}
where last inequality is due to  $\|\nabla f(X_k)\|\geq \zeta \Delta_k$. 

 \vskip0.5cm
 We now derive the bounds on   the expectation of $\Phi_{k+1}-\Phi_{k}$  in the remaining case. The proof of this case is different than  that of Case 2 of Theorem 4.11   \cite{ChenMenickellyScheinberg2014}, because of the dropped bound on $\eta_2$. 
 
 \paragraph{Case 2: $\| \nabla f(x_k)  \| < \zeta \delta_k $
with  $\zeta$ satisfying \eqref{eq:zeta}. }


%

First we note that if  $\| g_k \| <\eta_2\delta_k$, then we have an unsuccessful step and \eqref{eqn.unsuc.phi}  holds. Hence, we now assume that $\| g_k \| \geq \eta_2\delta_k$.  Here we consider only two outcomes, in particular, we will show that  when the estimates are good, \eqref{eqn.unsuc.phi}   holds.
 Otherwise, because $\| \nabla f(x_k)  \| < \zeta \delta_k $,
 the increase in $\phi_k$ can be bounded from above by a multiple of $\delta_k^2$. Hence by selecting appropriate value for probability  $\beta$ we will be able to establish the bound on expected decrease in $\Phi_k$ as in Case 1. 
\begin{itemize}
\item[a.] 
 $J_k=1$, i.e.,  the  estimates are good on iteration $k$, while the model might be good or bad.  

The iteration may or may not be successful. On successful iterations, the good estimates ensure reduction in $f$, while on unsuccessful iterations, $\delta_k$ is reduces. Applying the same argument as in the Case 1(c) we  establish  that  \eqref{eqn.unsuc.phi} always holds.

\item[b.] $J_k=0$, i.e.,  the  estimates are bad on iteration $k$, while the model might be good or bad.

Here, as in Case 1, we bound the maximum possible increase in $\phi_k$. Using the  Taylor expansion, the  Lipschitz continuity of $\nabla f(x)$ and  taking into account the bound $ \|\nabla f(x_k)\|<\zeta\delta_k$ we have
\begin{eqnarray*}
 f(x_k+s_k)-f(x_k)\le \|\nabla f(x_k)\|\delta_k+\frac{1}{2}L\delta_k^2<C_3\zeta\delta_k^2.
 \end{eqnarray*}
 Hence, the change in function $\phi$ is 
\begin{eqnarray}
\phi_{k+1}-\phi_k\le[\nu C_3\zeta+(1-\nu)(\gamma^2-1)]\delta_k^2.\label{eqn.suc.phi.falseincrease.case2}
\end{eqnarray}

\end{itemize}

We are now ready to bound the expectation of $\phi_{k+1}-\phi_{k}$  as we did in Case 1 except that now we simply combine 
(\ref{eqn.suc.phi.falseincrease.case2}),
which holds with probability at most $(1-\beta)$, and  (\ref{eqn.unsuc.phi}), which holds otherwise:
\begin{eqnarray}\label{eq:boundcase2}
&&\E[\Phi_{k+1}-\Phi_{k}|\mathcal{F}_{k-1}^{M\cdot F}, \{\| \nabla f(X_k)  \| < \zeta \Delta_k\} ]\nonumber \\
&\le&\beta(1-\nu)(\frac{1}{\gamma^2}-1)\Delta_k^2 \\
&&+ (1-\beta)[\nu C_3\zeta+(1-\nu)(\gamma^2-1)]\Delta_k^2\nonumber 
\end{eqnarray}
If we choose probability $0<\beta\leq 1$  so that the following holds,
\begin{equation}\label{eq:cond8}
\frac {\beta}{1-\beta} \ge \frac {2\nu \gamma^2 C_3\zeta } {(1-\nu)(\gamma^2 -1)} + 2\gamma^2,
\end{equation}
then the first term in \eqref{eq:boundcase2}, which is negative, 
 is at least twice as large in absolute value as the second term, which is positive. 
We thus have
\begin{eqnarray}\label{eq:final.decrease.in.phi.2}
 \E[\Phi_{k+1}-\Phi_{k}|\mathcal{F}_{k-1}^{M\cdot F}, \{\| \nabla f(X_k)  \| < \zeta \Delta_k\} ] \leq \frac{1}{2}\beta(1-\nu)(\frac{1}{\gamma^2}-1)\Delta_k^2.
\end{eqnarray}

 \vskip10mm

To complete   the proof of the lemma it remains to  substitute the appropriate constants in the above expressions. 
In particular, because of our assumptions that $\kappa_{bhm}\leq 12\kappa_{ef}$ and $\kappa_{ef} =\kappa_{eg}$, we can choose  $C_1=\frac{1}{10}$, and, recalling the choice of $\zeta=20\kappa_{eg}$, $\gamma=2$, $\eta_1=0.1$, $\kappa_{fcd}=0.5$ and $\eta_2\leq \kappa_{eg}$,  \eqref{eq:nu} reduces to
\begin{equation}\label{eq:nu_bound}
\frac{\nu}{1-\nu}\geq \frac{4\gamma^2}{\eta_1\eta_2\kappa_{fcd}}\geq  \frac{320}{ \eta_2},
\end{equation}
which holds if 
$$
\nu \geq \frac{320}{ 320+\eta_2}. 
$$
We can assume that $\nu > \frac{1}{2}$ without loss of generality. 

{\bf Case 1:}
For the  probabilities $\alpha$ and $\beta$ to satisfy \eqref{eq:alpha.beta.condition.1} with $C_3=1+\frac{3L}{2\zeta}$, it is sufficient that
$$
\frac{(\alpha\beta-\frac{1}{2})}{(1-\alpha)(1-\beta)}\ge 10 +\frac{30L}{40\kappa_{eg}}.
$$
  Then, using $\nu>\frac{1}{2}$ \eqref{eq:expectPhi1} implies 
$$
 \E[\Phi_{k+1}-\Phi_{k}|\mathcal{F}_{k-1}^{M\cdot F}, \{ \nabla f(X_k)  \| < \zeta \Delta_k\}] \le - \frac{1}{1600\kappa_{eg}}  \Delta_k^2.
$$

{\bf Case 2:} Recalling the expression for $C_3$ and the values for constant $\nu$, $\zeta$ and  $\gamma= 2$, and choosing $\nu$ so that
\eqref{eq:nu_bound} is satisfied with equality, we see that 
\eqref{eq:cond8} is satisfied if 
$$
\frac {\beta}{(1-\beta)} \ge \frac{4\times 320(40 \kappa_{eg} +3L)}{3\eta_2} +8,
 $$
 which is satisified if 
\begin{equation}\label{eq:cond8withconstants}
\frac {\beta}{(1-\beta)} \ge \frac {1280(14 \kappa_{eg}+L) }{\eta_2} +8,
\end{equation}
which  in turn is satisfied by 
$$
\beta\geq \frac{2\cdot 10^4 \kappa_{eg} +1280 L +8{\eta_2}}{2\cdot 10^4 \kappa_{eg} +1280 L +9\eta_2}=
\frac{\kappa_{eg} +0.064 L +4\cdot 10^{-4}{\eta_2}}{ \kappa_{eg} + 0.064 L +4.5\cdot 10^{-4}\eta_2}. 
$$
Then, observing that $\nu$ is chosen so that $1-\nu=\frac{\eta_2}{320+\eta_2}$,  from \eqref{eq:final.decrease.in.phi.2}
and $\eta_2<320$ (as $\nu>\frac{1}{2}$),
\[
\E[\Phi_{k+1}-\Phi_{k}|\mathcal{F}_{k-1}^{M\cdot F}, \{\| \nabla f(X_k)  \| < \zeta \Delta_k\} ]\leq- \frac{3\eta_2}{8(320+\eta_2)}\beta\Delta_k^2\leq -\frac{1}{1800}\eta_2\beta\Delta_k^2.
\]

Hence we conclude that 

$$
\E[\Phi_{k+1}-\Phi_{k}|\mathcal{F}_{k-1}^{M\cdot F} ]  \leq  -\Theta \Delta_k^2
$$
for  $\Theta= \frac{1}{1800}\min \left\{\eta_2\beta,\kappa_{eg}^{-1}\right\}$, which completes the proof. 

 \end{proof}

The  almost-sure stationarity result follows immediately from Theorem \ref{thm:main} with the same proof as  in \cite{ChenMenickellyScheinberg2014}, but this time without the assumption $\eta_2\geq \kappa_{ef}$. 
\begin{theorem} \label{th:convergenceliminfnew} 
 Let  Assumptions \ref{ass:F} and  \ref{ass:models_estim} hold,
  and let $\alpha$ and $\beta$ satisfy conditions of Theorem \ref{thm:main}, then
 the sequence of random iterates generated by Algorithm \ref{algo:stodfosimple},
 $\{X_k\}$, almost surely satisfies 
 \[ \underset{k\to \infty}{\lim} \| \nabla f(X_k)  \|=0.\]
\end{theorem}

The validity of the Assumption \ref{ass:stoch-proc}(iii) follows from  Theorem \ref{thm:main}.  
We 
state the result below  for completeness and convenience of reference. 
\begin{lemma}\label{lem:Vk}
Let  the assumptions of Theorem \ref{thm:main} hold. 
Then Assumption \ref{ass:stoch-proc}(iii) is satisfied, with 
$\Theta= \frac{1}{1800}\min \left\{\eta_2\beta,\kappa_{eg}^{-1}\right\}$
for the process $\{\Phi_k, \Delta_k\}$, where $\Phi_k$ is defined as in \eqref{eq:Phidef} with $\nu$ satisfying 
\eqref{eq:nu} and $h(\delta)= \delta^2$.
\end{lemma}

\subsection{Complexity result for first order STORM algorithm}

\begin{theorem}\label{the:firstorder}
Consider Algorithm \ref{algo:stodfosimple} and the corresponding stochastic process.  Let $T_\epsilon$ be defined as in \eqref{eq:Tdef}. 
 Then, under the assumptions of Theorem \ref{thm:main}, 
\[ 
\E[T_\epsilon] \leq \frac{\alpha\beta}{2\alpha\beta-1}
 \left(\displaystyle\frac{20\Phi_0\kappa_{eg}}
{\Theta\epsilon^2}
+1\right),
\]
where 
$\Theta= \frac{1}{1800}\min \left\{\eta_2\beta,\kappa_{eg}^{-1}\right\}$,  $\Phi_0$ defined as in \eqref{eq:Phidef} with $k=0$, with $\nu$ satisfying 
\eqref{eq:nu}. 
\end{theorem}

\subsection{Example of models and estimates satisfying Assumption \ref{ass:models_estim}}
\label{1:Storm-examples}
While the assumption \ref{ass:models_estim},  which allows us to develop the general complexity analysis, is fairly general it is easy to satisfy in practice in the classical stochastic optimization setting by taking a sufficient number of samples of the function, gradient and Hessian estimates. 
A number of recent papers rely on this technique, for producing sufficiently accurate gradient and Hessian approximations. For example Lemma 4 in \cite{Tripuraneni2017} uses matrix concentration results from \cite{Tropp2015} to show that given the bound on the variance of the gradient 
$$
\E[ \|\nabla {\tilde f}(x,\xi)-\nabla f(x)\|] \leq \sigma_g^2,
 $$
 the average of $\tilde {\cal O}(\frac{\sigma_g^2}{\varepsilon^2})$ gradient samples $\nabla {\tilde f}(x,\xi)$, $g$, satisfies
 $$
 \|g-\nabla {\tilde f}(x,\xi)\|\leq \epsilon
 $$
 with probability $p$, where $\tilde {\cal O}$ hides the term dependent on $-\log(1-p)$. Similar result is established for the Hessian sample average approximation. The result for the function estimates, given the variance $\sigma_f$ is a simpler version of the same inequalities and can be derived sing Chebychev inequality. 
 
 Using these results, we can obtain $\alpha$-probabilistically fully-linear models as follows. 
 Compute $f_k$ by averaging  ${\cal O}(\frac{\sigma_f^2}{\Delta_k^4}\log(\frac{1}{1-\sqrt{\alpha}}))$ samples  ${\tilde f}(x_k,\xi)$, then compute
  $g_k$ as an average of $\tilde {\cal O}(\frac{\sigma_g^2}{\kappa_{eg}^2\Delta_k^2}\log(\frac{1}{1-\sqrt{\alpha}}))$ gradient samples $\nabla {\tilde f}(x_k,\xi)$. This ensures that $\|g_k-\nabla f(x_k)\|\leq \kappa_{eg}\Delta_k$ and $|f_k-f(x_k)|\leq \Delta_k^2$ with probability at least $\alpha$.  Condition 
 $|m(y)- f(y)|\leq \kappa_{ef}\Delta_k^2$ follows automatically  with appropriately chosen $\kappa_{ef}$. 
 
 We want to note that all sample sizes are determined by quantities that are generally either chosen or known by the algorithm or can be correctly estimated.  
 
 Similarly,  we  can  obtain  $\beta$-probabilistically $\epsilon_F$-accurate estimates $f_k^0$ and $f_k^s$ by averaging 
 $\tilde {\cal O}(\frac{\sigma_f^2}{\epsilon_F^2\Delta_k^4}\log(\frac{1}{1-\sqrt{\beta}}))$ samples of $ {\tilde f}(x_k,\xi)$ and $ {\tilde f}(x_k+s_k,\xi)$ for each., respectively.

In the case of simulation optimization, when $\nabla f(x,\xi)$ is not available, $\kappa$-fully-linear models $m_k$ can be constructed via polynomial interpolation \cite{DFObook}, and $\alpha$-probabilistically $\kappa$-fully-linear models are similarly obtained by combining interpolation and sufficiently accurate function value estimates (see, e.g. \cite{2015sarhasetal}).


Another setting that is explored in  \cite{ChenMenickellyScheinberg2014} is when $f(x)$ and, possibly, $\nabla \tilde f(x)$ are computed (accurately) via some procedure, which may fail with some small, but fixed, probability. In this case $\tilde f(x,\xi)$ and $\nabla \tilde f(x, \xi)$ are the true values of the function and the gradients, or some arbitrarily corrupted values. If the probability of failure is sufficiently small, conditioned on the past, the STORM algorithm still converges almost surely. See  \cite{ChenMenickellyScheinberg2014}  for details.


\section{The second order STORM algorithm}
\label{section:second-order}

We now introduce a variant of Algorithm \ref{algo:stodfosimple} that attempts to achieve second order criticality in 
the stochastic setting; we use the same notation as in Algorithm \ref{algo:stodfosimple}.
Firstly, 
the model minimization may need to provide more than just the Cauchy decrease \eqref{eqn:CS},
namely, we require that on each iteration $k$ and for all model realizations $m_k$ (as defined in Step 2)
 of $M_k$, we are able to compute a step
$s_k$, so that the following level of second order
improvement is achieved,
\begin{equation}\label{fod-final} m_k(x_k) - m_k(x_k+s_k) \; \geq \;
\frac{\kappa_{scd}}{2} \max \left\{
            \|g_k\|
            \min\left[ \frac{\|g_k\|}{\| H_k \|}, \delta_k \right],
           \max \{-\lambda_{\min}(H_k), 0\} \delta_k^2
           \right\}.
 \end{equation} for some constant $\kappa_{scd}
\in (0,1]$. A step satisfying this (typical second order) assumption is given, for instance, by
computing both the Cauchy step and,
in the presence of negative curvature in the model,
the {\it eigenstep}, and by choosing the
one that provides the largest reduction in the model\footnote{The eigenstep is the minimizer of the quadratic model in the trust region along an eigenvector
corresponding to the smallest (negative) eigenvalue of~$H_k$.} \cite{TRBook}.

In our analysis (not in the algorithm), we will use --- instead of just the true gradient of $f$ --- the following measure of proximity to a second order stationary point for the objective $f$, 
\begin{equation}\label{2:f-optimality}
\tau(x) \;=\;
\max \left\{\| \nabla f(x) \|, 
 -\lambda_{\min}(\nabla^2f(x)) \right\}.
\end{equation}
The corresponding optimality measure for the model $m_k$ is defined slightly
differently than above, following \cite{bandeira2013convergence},
\begin{equation}\label{2:model-optimality}
\tau^m(x) \;=\;
\max \left\{ \min \left[ \| \nabla m(x) \|, \frac{\| \nabla m(x) \|} {\| \nabla^2 m(x) \|}\right],
-\lambda_{\min}(\nabla^2m(x)) \right\}.
\end{equation}
The additional term in \eqref{2:model-optimality} compared to \eqref{2:f-optimality}  is
needed because there is no longer a bound on  the model Hessians on
all iterations,
as  in the first order case. We will only use \eqref{2:model-optimality} at the iterate $x_k$, in which case, it becomes
\begin{equation}\label{2:model-optimality-2}
\tau^m_k \;=\;
\max \left\{ \min \left[ \| g_k \|, \frac{\| g_k \|} {\| H_k \|}\right],
-\lambda_{\min}(H_k) \right\}.
\end{equation}

We are now ready to present our second order STORM algorithm, by modifying the first order STORM algorithm. 

\begin{algorithm}
\caption{{\sc second order Stochastic DFO with Random Models}}
\label{algo:stodfosecond}
Perform the STORM Algorithm \ref{algo:stodfosimple}, with the following modifications to Steps 3, 5 and 6:
 \begin{itemize}
\item[3:] (Step calculation) Compute $s_k = \arg \underset{s: \| s \|\leq \delta_k}{\min}  m_k(s) $ (approximately) so that it satisfies  condition \eqref{fod-final}. 
\item[5:] (Acceptance of the trial point): If $\rho_k\geq \eta_1$ and $\tau^m_k \geq \eta_2 \delta_k$, set $x_{k+1} =x_k+s_k$; otherwise, set $x_{k+1}=x_k$.
\item[6:] (Trust-region  radius update): If $\rho_k\geq \eta_1$ and $\tau^m_k \geq \eta_2 \delta_k$, set $\delta_{k+1} =\min\{  \gamma \delta_k,\delta_{\max}\}$; otherwise $\delta_{k+1}=\gamma^{-1} \delta_k$; $k \gets k+1$ and go to step \ref{step.model}.
  \end{itemize}
\end{algorithm}

The analysis for the second order STORM variant will again use the framework proposed
in Section 2, thus serving as another illustration of the applicability of our generic set up. Before
applying this framework we need to describe our assumptions required for a second order analysis.

In terms of problem assumptions,  we will need one more order of
smoothness compared to first order ones (Assumption \ref{ass:F}).

\begin{assumption} \label{assumptions:Lip2_cont}
Assume that $f$ satisfies Assumption \ref{ass:F} and that it is twice
continuously differentiable on ${\cal X}$, and also that the
Hessian $\nabla^2 f$  is $L_H$-Lipschitz continuous. 
\end{assumption}


\subsection{Assumptions on the second order STORM algorithm}


Let us now introduce a measure of second order quality or accuracy of the models $m_k$
(see~~\cite{DFObook, BillupsLarson, larson2013stochastic} for more details).

\begin{definition}
1) A function $m_k$ is a \emph{$\kappa$-fully quadratic model of $f$ on $B(x_k,\d_k)$} provided, for $ \kappa = (\kappa_{ef}, \kappa_{eg},\kappa_{eh})$ and  $\forall y \in B(x_k,\delta_k)$,
\begin{eqnarray*}\label{eq:fq-model}
\|\nabla^2 f(x_k) - H_k \| & \leq & \kappa_{eh} \delta_k, \\
\|\nabla f(y) - \nabla m_k(y) \| & \leq & \kappa_{eg} \delta_k^2, \\
\|f(y) - m(y) \| & \leq & \kappa_{ef} \delta_k^3.
\end{eqnarray*}

2) The estimates $f_k^0$ and $f_k^s$ are said to be $\epsilon_F$-s.o.-accurate (s.o. for "second order") estimates of $f(x_k)$ and $f(x_k+s_k)$, respectively,  for a given $\delta_k$ if
\begin{equation}\label{eq:estimates_2}
 |f_k^0 - f(x_k) |  \le \epsilon_F \d_k^3 \mbox{ and } |f_k^s - f(x_k+s_k) | \le \epsilon_F \d_k^3.
 \end{equation}

\end{definition} 

\begin{definition}\label{def:pfq-models}
\cite{ASBandeira_LNVicente_KScheinberg_2013}
A sequence of random models $\{ M_k \}$ is said to be $\alpha$-{\bf probabilistically} $\kappa$-{\bf fully quadratic} with respect to the corresponding sequence $\{ B(X_k,\Delta_k)\}$ if the events 
\begin{equation}\label{eq:Ik_2}
I_k = \one\{ M_k \mbox{ is a } \kappa \mbox{-fully quadratic model of } f \mbox{ on }B(X_k,\Delta_k)  \}
\end{equation}
satisfy the condition
\[ P (I_k=1 | \mathcal{F}_{k-1}^{M\cdot F})\ge \alpha.\]
\label{probabilisticmodel2}
\end{definition}
\begin{definition}\label{def:pa-estimates_2} A sequence of random estimates $\{F_k^0,F_k^s\}$ is said to be $\beta$-{\bf probabilistically} $\epsilon_F$-{\bf s.o.-accurate} with respect to the corresponding sequence $\{X_k,\Delta_k,S_k\}$ if the events
\begin{equation}\label{eq:Jk_2}
 J_k =\one \{F_k^0,F_k^s \mbox{ are } \epsilon_F\mbox{-s.o.accurate estimates of }f(x_k) \mbox{ and }f(x_k+s_k), \mbox{ respectively, \ for\ } \Delta_k  \}  
 \end{equation}
satisfy the condition 
\[P( J_k=1|\mathcal{F}_{k-1/2}^{M\cdot F}  )\ge \beta,\]
where $\epsilon_F$ is a fixed constant. 
\label{asmpesti2}
\end{definition}

We will no longer assume that the Hessian $H_k$ of the models is bounded in norm, since we cannot simply disregard
large Hessian model values without possibly affecting the chances of the model being fully quadratic. However, a simple analysis can show that $\|H_k\|$ is uniformly bounded from above for any
fully quadratic model $m_k$ (although we may not know what this bound is and hence may not be able to use it in an algorithm).

\begin{lemma}\cite{bandeira2013convergence}\label{lemma:bmh}
Let Assumption \ref{assumptions:Lip2_cont} hold.
Given constants $\kappa_{eh}$, $\kappa_{eg}$, $\kappa_{ef}$,  and $\delta_{\max}$,
there exists a constant $\kappa_{bhm}\geq 1$ such that
for every~$k$ and every realization~$m_k$ of $M_k$
which is a $(\kappa_{ef},\kappa_{eg},\kappa_{eh})$-fully quadratic
model of $f$ on $B(x_k,\delta_k)$ with $x_k\in {\cal X}$ and $\delta_k\leq \delta_{\max}$ we have
\[
\|H_k\| \; \leq \; \kappa_{bhm}.
\]
\end{lemma}

\proof{The proof follows trivially from the definition of fully quadratic models and the assumption that
$\|\nabla^2 f\|\leq L$ is bounded above  on
${\cal X}$, which follows from the gradient of $f$ being Lipschitz continuous with constant $L$. Then we can let $\kappa_{bhm}:=\delta_{\max}\kappa_{eh}+L$.}

For our convergence analysis we again need to impose conditions on the nature of the stochastic (and deterministic) information used by our algorithm.
\begin{assumption}\label{ass:models_estim2} The following hold for the quantities used in Algorithm \ref{algo:stodfosecond}

\begin{itemize}

\item[(a)] The sequence of random models $M_k$, generated by Algorithm
  \ref{algo:stodfosecond}, is $\alpha$-probabilistically
  $\kappa$-fully quadratic, for some 
$\kappa= (\kappa_{ef}, \kappa_{eg},\kappa_{eh})$ and for a sufficiently large $\alpha\in (0,1)$. 

\item[(b)] The sequence of random estimates $\{F_k^0,F_k^s\}$ generated by Algorithm \ref{algo:stodfosecond} is 
 $\beta$-probabilistically $\epsilon_F$-s.o.accurate 
for   $\epsilon_F\leq \kappa_{ef}$ and $\epsilon_F< \frac{1}{4}\eta_1\eta_2\kappa_{scd}\min\{\eta_2,1\}$, and for a sufficiently
large $\beta\in (0,1)$. 
\end{itemize}
\end{assumption}

Note that as in the first order case, we are able to allow for unrestricted values of $\eta_2$ in Algorithm \ref{algo:stodfosecond}, with a potential trade-off of increased accuracy on the function estimates.

\subsection{Useful preliminary results for second order STORM analysis}

The analysis of Algorithm \ref{algo:stodfosecond} is similar to that
of the first order STORM described in Section~\ref{sec:firstorder}.
However, there are more cases to consider and the convergence rate to
the second order stationary point is different, as it is in the
deterministic case. There will also be another significant difference, such as a requirement for an additional assumption on
function estimates,
to be detailed in the next section.
First, we state and prove the analogues of Lemmas \ref{lemma.delta.1}--\ref{lemma.delta.4} for the function decrease in terms of first and second order optimality. The first three lemmas are almost identical to Lemmas \ref{lemma.delta.1}--\ref{lemma.delta.3}, except that the models are assumed to be fully quadratic, instead of fully linear; the model decrease condition \eqref{fod-final} is now used and  condition
$\|H_k\|\leq \kappa_{bhm}$ is only valid when the model $m_k$ is
fully-quadratic according to Lemma \ref{lemma:bmh}. For completeness,
we have delegated the proofs of Lemmas \ref{lemma.delta.1_1}--Lemmas
\ref{lemma.delta.3_1} to the Appendix.

\begin{lemma}\label{lemma.delta.1_1} 
{\rm [Good quadratic model $\Rightarrow$ function reduction in $\|g_k\|$]}
Let Assumption \ref{assumptions:Lip2_cont} hold. Suppose that a model $m_k$  is a
$(\kappa_{ef},\kappa_{eg},\kappa_{eh})$-fully quadratic model of $f$
on $B(x_k,\d_k)$. If $\delta_k\leq 1$ and
$$\d_k\le\min\left\{ \frac{1}{\kappa_{bhm}},\frac{\kappa_{scd}}{8\kappa_{ef}}  \right\} \|g_k\|,$$
then the trial step $s_k$ leads to an improvement in $f(x_k+s_k)$ such that
\begin{eqnarray*}
f(x_k+s_k)-f(x_k)\le - \frac{\kappa_{scd}}{4} \|g_k\|\d_k .
\end{eqnarray*}
\end{lemma}

\begin{lemma}\label{lemma.delta.2_1} 
{\rm [Good quadratic model $\Rightarrow$ function reduction in $\|\nabla f(x_k)\|$]}
Let Assumption \ref{assumptions:Lip2_cont} hold.
Suppose that a model is
  $(\kappa_{ef},\kappa_{eg},\kappa_{eh})$-fully quadratic on
  $B(x_k,\d_k)$. If $\delta_k\leq 1$ and
\begin{eqnarray}\label{condition.lemma.delta.2_1}
\d_k\le \min \left\{ \frac{1}{\kappa_{bhm}+\kappa_{eg}}, \frac{1}{\frac{8\kappa_{ef}}{\kappa_{scd}}+\kappa_{eg}}\right\} \|\nabla f(x_k)\|,
\end{eqnarray}
then the trial step $s_k$ leads to an improvement in $f(x_k+s_k)$ such that
\begin{eqnarray}f(x_k+s_k)-f(x_k)\le - C_1 \|\nabla f(x_k)\|\d_k , \label{eqn.true.decrease_1}
\end{eqnarray}
for any $C_1\leq\frac{\kappa_{scd}}{4}\cdot \max\left\{ \frac{\kappa_{bhm}}{\kappa_{bhm}+\kappa_{eg}},\frac{8\kappa_{ef}}{8\kappa_{ef}+\kappa_{scd}\kappa_{eg}}\right\}.$
\end{lemma}

\begin{lemma}\label{lemma.delta.3_1} 
{\rm [Good quadratic model + good s.o.~estimates $\Rightarrow$ successful step]}
Let Assumption \ref{assumptions:Lip2_cont} hold.
Suppose that $m_k$ is
  $(\kappa_{ef},\kappa_{eg},\kappa_{eh})$-fully quadratic on
  $B(x_k,\d_k)$ and the estimates $\{ f_k^0,f_k^s \}$ are
  $\epsilon_F$-s.o. accurate with $\epsilon_F\le \kappa_{ef}$. If
  $\delta_k\leq 1$ and
\begin{eqnarray}\label{condition.lemma.delta.3_1}
 \d_k \le \min \left\{ \dfrac{1}{\kappa_{bhm}} , \frac{1}{\eta_2\kappa_{bhm}}, \dfrac{ \kappa_{scd}  (1-\eta_1)}{ 8\kappa_{ef} }   \right\} \| g_k \| , 
 \end{eqnarray}
then the $k$-th iteration is successful.
\end{lemma}

The remaining lemmas address the case of negative curvature in the model and that of second order accurate estimates.

\begin{lemma}\label{lemma.delta.1_2}
{\rm [Good quadratic model $\Rightarrow$ function reduction in $\lambda_{\min}(H_k)$]}
Let Assumption \ref{assumptions:Lip2_cont} hold.
Suppose that a model $m_k$  is a
$(\kappa_{ef},\kappa_{eg},\kappa_{eh})$-fully quadratic model of $f$
on $B(x_k,\d_k)$. If  
\begin{equation}\label{2:delta-minH}
\d_k\le\frac{\kappa_{scd}}{8\kappa_{ef}} (-\lambda_{\min}(H_k)),
\end{equation}
then the trial step $s_k$ leads to an improvement in $f(x_k+s_k)$ such that
\begin{eqnarray}f(x_k+s_k)-f(x_k)\le - \frac{\kappa_{scd}}{4} (-\lambda_{min}(H_k))\d_k^2 .
\end{eqnarray}
\end{lemma}

\begin{proof} 
Whenever $\lambda_{\min}(H_k)<0$, 
the optimal decrease condition \eqref{fod-final} ensures that
\begin{eqnarray*}
m_k(x_k)-m_k(x_k+s_k) \ge \frac{\kappa_{scd}}{2} (-\lambda_{min}(H_k))\d_k^2 .
\end{eqnarray*}

Since the model is $\kappa$-fully quadratic, the improvement in $f$ achieved by $s_k$ is
\begin{eqnarray*}
&&f(x_k+s_k)-f(x_k)\\
&=& f(x_k+s_k)-m(x_k+s_k)+m(x_k+s_k)-m(x_k)+m(x_k)-f(x_k)\\
&\le &2\kappa_{ef}\d_k^3-\frac{\kappa_{scd}}{2} (-\lambda_{min}(H_k))\delta_k^2\\
&\le & - \frac{\kappa_{scd}}{4} (-\lambda_{min}(H_k))\delta_k^2
\end{eqnarray*}
where the last inequality is implied by  \eqref{2:delta-minH}.
\hfill\end{proof}

\begin{lemma}\label{lemma.delta.2_2} 
{\rm [Good quadratic model $\Rightarrow$ function reduction in $\lambda_{\min}(\nabla^2 f(x_k))$]}
Let Assumption \ref{assumptions:Lip2_cont} hold.
Suppose that a model is
  $(\kappa_{ef},\kappa_{eg},\kappa_{eh})$-fully quadratic on
  $B(x_k,\d_k)$. If  
\begin{eqnarray}\label{condition.lemma.delta.2_2}
\d_k\le  \frac{1}{\frac{8\kappa_{ef}}{\kappa_{scd}}+\kappa_{eh}}  (-\lambda_{min}(\nabla^2f(x_k))),
\end{eqnarray}
then the trial step $s_k$ leads to an improvement in $f(x_k+s_k)$ such that
\begin{eqnarray}f(x_k+s_k)-f(x_k)\le - C_4 (-\lambda_{min}(\nabla^2f(x_k)))\d^2_k , \label{eqn.true.decrease_2}
\end{eqnarray}
for any $C_4\leq \frac{\kappa_{scd}}{4}\cdot \frac{8\kappa_{ef}}{8\kappa_{ef}+\kappa_{scd}\kappa_{eh}}.$
\end{lemma}

\begin{proof}
The definition of a $\kappa$-fully-quadratic model, by Corollary 8.5.6 from \cite{GolubVanLoan} yield that
\begin{equation}\label{2:BvL}
-\lambda_{min}(H_k)\ge (-\lambda_{min}(\nabla^2f(x_k)))-\kappa_{eh}\delta_k.
\end{equation}
Since condition (\ref{condition.lemma.delta.2_2})  implies that $-\lambda_{min}(\nabla^2f(x_k)) \ge(\frac{8\kappa_{ef}}{\kappa_{scd}}+\kappa_{eh})\d_k$,  we have
$$-\lambda_{min}(H_k) \ge  \frac{8\kappa_{ef}}{\kappa_{scd}} \d_k.$$
Hence, the conditions of Lemma \ref{lemma.delta.1_2} hold and we have 
\begin{eqnarray}\label{lemma.delta.2.eqn.1_2}
f(x_k+s_k)-f(x_k)\le  - \frac{\kappa_{scd}}{4} (-\lambda_{min}(H_k))\d^2_k .
\end{eqnarray}
From \eqref{2:BvL} and (\ref{condition.lemma.delta.2_2}), we also have 
\begin{eqnarray}\label{lemma.delta.2.eqn.2_2}
-\lambda_{min}(H_k) \ge  \frac{8\kappa_{ef}}{8\kappa_{ef}+\kappa_{scd}\kappa_{eg}}(-\lambda_{min}(\nabla^2f(x_k)))
\end{eqnarray}
Combining (\ref{lemma.delta.2.eqn.1_2}) and (\ref{lemma.delta.2.eqn.2_2}) yields (\ref{eqn.true.decrease_2}).
\hfill\end{proof}

\begin{lemma}\label{lemma.delta.3_2} 
{\rm [Good quadratic model + good s.o.~estimates $\Rightarrow$ successful step]}
Let Assumption \ref{assumptions:Lip2_cont} hold.
Suppose that $m_k$ is $(\kappa_{ef},\kappa_{eg},\kappa_{eh})$-fully quadratic on $B(x_k,\d_k)$ and the estimates $\{ f_k^0,f_k^s \}$ are $\epsilon_F$-s.o.accurate with $\epsilon_F\le \kappa_{ef}$. If  
\begin{eqnarray}\label{condition.lemma.delta.3_2}
 \d_k \le  \min\left\{\frac{1}{\eta_2}, \dfrac{ \kappa_{scd}  (1-\eta_1)}{ 8\kappa_{ef} } \right\} (-\lambda_{min}(H_k)), 
 \end{eqnarray}
then the $k$-th iteration is successful.
\end{lemma}

\begin{proof}
From the model decrease condition \eqref{fod-final}  
\begin{eqnarray}
m_k(x_k) -m_k(x_k+s_k) \ge  \dfrac{\kappa_{scd}}{2}(-\lambda_{min}(H_k))\d_k^2.\label{eqn:deltainc1_2}
\end{eqnarray}
The model $m_k$ being  $(\kappa_{ef},\kappa_{eg})$-fully quadratic implies that
\begin{eqnarray}
|f(x_k)-m_k(x_k)|&\le& \kappa_{ef}\d_{k}^3,\mbox{ and } \label{eqn:deltainc2_2}\\
|f(x_k+s_k)-m_k(x_k+s_k)|&\le& \kappa_{ef}\d_{k}^3.\label{eqn:deltainc3_2}
\end{eqnarray}
Since the estimates are $\epsilon_F$-s.o.accurate with $\epsilon_F\le \kappa_{ef}$, we obtain\begin{eqnarray}
|f_k^0-f(x_k)|\le \kappa_{ef}\d_k^3,\mbox{ and } |f_k^s-f(x_k+s_k)|\le \kappa_{ef}\d_k^3.\label{eqn:deltainc4_2}
\end{eqnarray}

We have
\begin{eqnarray*}
\rho_k &= &\dfrac{f_k^0-f_k^s}{m_k(x_k)-m_k(x_k+s_k)}\\
&=&  \dfrac{f_k^0-f(x_k) }{m_k(x_k)-m_k(x_k+s_k)}+\dfrac{f(x_k) -m_k(x_k)}{m_k(x_k)-m_k(x_k+s_k)}+\dfrac{m_k(x_k)-m_k(x_k+s_k) }{m_k(x_k)-m_k(x_k+s_k)}\\
&& + \dfrac{m_k(x_k+s_k)-f(x_k+s_k)}{m_k(x_k)-m_k(x_k+s_k)}+\dfrac{f(x_k+s_k)   -f_k^s   }{m_k(x_k)-m_k(x_k+s_k)},
\end{eqnarray*}
which, combined with (\ref{eqn:deltainc1_2})-(\ref{eqn:deltainc4_2}), implies
\begin{eqnarray*}
| \rho_k-1|\le \dfrac{8\kappa_{ef}\d_k^3}{\kappa_{scd}(-\lambda_{min}(H_k))\d_k^2 }\le 1-\eta_1,
\end{eqnarray*}
where we have used the assumptions $ \d_k \le \frac{ \kappa_{scd}(1-\eta_1) }{ 8\kappa_{ef}  }(-\lambda_{min}(H_k))$ to deduce the last inequality. Hence, $\rho_k\ge\eta_1$. Moreover, the first term in 
\eqref{condition.lemma.delta.3_2} and \eqref{2:model-optimality-2} imply $\tau_k^m\geq(-\lambda_{min}(H_k)) \geq \eta_2\delta_k$. Thus the $k$-th iteration is successful.
\hfill\end{proof}

\begin{lemma}\label{lemma.delta.4_2}
{\rm [Good s.o.~estimates + successful step $\Rightarrow$ function decrease in $\delta_k^3$]}
Assume the estimates $\{f_k^0,f_k^s  \}$ are
$\epsilon_F$-s.o.accurate with $\epsilon_F< \frac{1}{4}\eta_1\eta_2\min\{1,\eta_2\}
\kappa_{scd}$. If $\delta_k\leq 1$ and a trial step $s_k$ is accepted (a successful iteration occurs), then the improvement in $f$ is bounded below as follows
\begin{eqnarray}\label{eqn.lemma.4_2}
f(x_{k+1})-f(x_k)\le -C_2\delta_k^3,\end{eqnarray}
where 
\begin{equation}\label{eq:C2_2}
C_2 =\frac{1}{2}\eta_1\eta_2\min\{1,\eta_2\} \kappa_{scd}-2\epsilon_F>0.
\end{equation}
\end{lemma}
\begin{proof} An iteration being successful indicates that $\rho\ge \eta_1$ and either $\min \left\{ \|g_k\|, \frac{\| g_k\| }{\| H_k \|} \right \}\ge\eta_2 \d_k$ or $-\lambda_{min}(H_k)\geq \eta_2 \d_k$. 
First let us assume that $\min \left\{ \|g_k\|, \frac{\| g_k\| }{\| H_k \|} \right \}\ge\eta_2 \d_k$; then 
\begin{eqnarray*}
f_k^0-f_k^s &\ge&\eta_1(m_k(x_k)-m_k(x_k+s_k))\\
&\ge&\eta_1 \frac{\kappa_{scd}}{2} \| g_k \| \min \left\{ \frac{\| g_k\| }{\| H_k \|} ,\delta_k\right\}\\
&\ge& \frac{1}{2} \eta_1 \kappa_{scd}\eta_2\min\{1,\eta_2\}\delta_k^2\\
&\ge& \frac{1}{2} \eta_1 \kappa_{scd}\eta_2\min\{1,\eta_2\}\delta_k^3,
\end{eqnarray*}
where we also used $\delta_k\leq 1$.

Let us now assume that $-\lambda_{min}(H_k)\geq \eta_2 \d_k$,  thus,
\begin{eqnarray*}
f_k^0-f_k^s &\ge&\eta_1(m_k(x_k)-m_k(x_k+s_k))\\
&\ge&\eta_1 \frac{\kappa_{scd}}{2} ( -\lambda_{min}(H_k))\delta_k^2\\
&\ge& \frac{1}{2} \eta_1\eta_2 \kappa_{scd}\delta_k^3\\
&\ge& \frac{1}{2} \eta_1 \kappa_{scd}\eta_2\min\{1,\eta_2\}\delta_k^3.
\end{eqnarray*}

Thus in both cases, using the fact that the estimates are $\epsilon_F$-s.o.accurate, we have 
\begin{eqnarray*}
f(x_k+s_k)-f(x_k)= f(x_k+s_k)-f_k^s+f_k^s-f_k^0+f_k^0-f(x_k)
\le  -C_2\delta_k^3,
\end{eqnarray*}
here $C_2$ is defined in \eqref{eq:C2_2}.
\hfill\end{proof}

\paragraph{Choosing constants}\label{page:constants_2}
To simplify our calculations, just like for the first order case, we particularize our choices of constants, but we will
clearly state  when we use these choices. We let $\kappa_{scd}=0.5$, $\eta_1=0.1$,  $\gamma=2$, $\delta_{\max}=1$ and 
$\kappa_{ef}=\kappa_{eg}=\kappa_{eh}=\Theta(\overline{L})$, where $\overline{L}=\max\{L,L_H\}$. To satisfy 
Assumption \ref{ass:models_estim2}, we let $\epsilon_F=\frac{1}{160}\eta_2\min\{1,\eta_2\}\leq \kappa_{eh}$ and
$\eta_2\leq 18$. Note that we cannot impose upper bounds on $\kappa_{bhm}$ as the latter cannot be chosen freely, namely, 
from Lemma \ref{lemma:bmh}, we have $\kappa_{bhm}=\kappa_{eh}+L\leq 2\max\{\kappa_{eh}, \overline{L}\}$.

\subsection{Defining and analysing the process $\{\Phi_k,\Delta_k\}$ for
the second order case}

As the order of the function decrease that can be guaranteed on good
iterations of Algorithm \ref{algo:stodfosecond}
changes from the first order $\delta_k^2$ to $\delta_k^3$ due to
second order terms, we must modify the process $\Phi_k$ accordingly.
Namely, we let $\{\Phi_k, \Delta_k\}$ be derived from the  process  generated by Algorithm \ref{algo:stodfosecond}, 
with $\Delta_k$ - the trust region radius and 
\begin{equation}\label{eq:Phidef_2}
\Phi_k = \nu f(x_k) + (1-\nu)\Delta_k^3
\end{equation}
where $\nu\in (0,1)$ is a deterministic, large enough constant, which we will define later, and  $\Phi_k\geq 0$. 
We also define the  random time 
\begin{equation}\label{eq:Tdef2}
T_{\epsilon} = \inf\{k\geq 0 : \|\nabla f(X_k)\|\leq\epsilon \quad{\rm
  and}\quad \lambda_{\min}(\nabla^2 f(X_k))\geq -\epsilon\},
\end{equation} 
which is a stopping time for the stochastic process defined by Algorithm \ref{algo:stodfosecond} and hence for $\{\Phi_k, \Delta_k\}$ as in \eqref{eq:Phidef_2}.
To bound the expected stopping time
$\mathbb{E}(T_{\epsilon})$ for Algorithm \ref{algo:stodfosecond}, we show that Assumption \ref{ass:stoch-proc} is satisfied for 
$\{\Phi_k, \Delta_k\}$ in \eqref{eq:Phidef_2} and apply the results of Section \ref{sec:walds}.

Very similarly to the first order case, we can show that  Assumption \ref{ass:stoch-proc}(i)--(ii) holds with $\lambda=\log\gamma$, and with
the following new settings
\begin{equation}\label{eq:zeta_2}
\Delta_\epsilon = \displaystyle\frac{\epsilon}{\zeta}, \quad {\rm for \  }\zeta\geq\max\{\kappa_{eg},\kappa_{eh}\}+\max\left\{\eta_2\kappa_{bhm}, \kappa_{bhm},\frac{8\kappa_{ef}}{\kappa_{scd}(1-\eta_1)}\right\},
\end{equation}
with $\epsilon \in (0,1]$, and the (old) assumption that $\Delta_\epsilon=\gamma^i \delta_0$ for some $i\leq 0$. 
Note that \eqref{eq:zeta_2}, $\epsilon \in (0,1]$ and $\kappa_{bhm}\geq 1$ imply that $\Delta_{\epsilon}\leq 1$.

\begin{lemma}\label{lem:Deltak2}
Let Assumptions \ref{assumptions:Lip2_cont} and  \ref{ass:models_estim2} hold. Let $\alpha$ and $\beta$ be such that $\alpha\beta\geq 1/2$, then Assumption \ref{ass:stoch-proc}(ii)
is satisfied for Algorithm \ref{algo:stodfosecond} with $W_k=2(I_kJ_k-\frac{1}{2})$, $\lambda=\log\gamma$ and $p=\alpha\beta$. 
\end{lemma}
\begin{proof}
The proof follows similarly to that of Lemma \ref{lem:Deltak} and we show that, conditioned on $T_\epsilon > k$ (i.e. $I(T_\epsilon > k)=1$), where $T_{\epsilon}$ is now defined in \eqref{eq:Tdef2}, \eqref{Deltakdynamics} holds with $\Delta_{\epsilon}$ defined
in \eqref{eq:zeta_2}. The only case that differs (from the first order proof) and needs addressing is when 
 $\Delta_k\leq\Delta_\epsilon$. Then, conditioned on
$T_\epsilon > k$, we have that either $\|\nabla f(X_k)\|\geq\epsilon$
or $\lambda_{\min}(\nabla^2 f(X_k))\leq -\epsilon$ and hence, from the
definition of $\zeta$ in \eqref{eq:zeta_2},  
we know that
\begin{equation}\label{eq:tmp:grad}
 \|\nabla f(x_k)\|\ge \left(\kappa_{eg}+ \max\left\{  \eta_2\kappa_{bhm},\kappa_{bhm},
     \frac{8\kappa_{ef}}{\kappa_{scd}(1-\eta_1)}\right\}\right)\delta_k
\end{equation}
or that 
\begin{equation}\label{eq:tmp:hess}
 -\lambda_{\min}(\nabla^2 f(X_k))\geq \left(\kappa_{eh}+ \max\left\{  \eta_2,
     \frac{8\kappa_{ef}}{\kappa_{scd}(1-\eta_1)}\right\}\right)\delta_k,
 \end{equation}
where we also used that $\kappa_{bhm}\geq 1$.
Assume that $I_k=1$ and $J_k=1$, i.e., both the  model and the
estimates are good on iteration $k$.  Since the model $m_k$ is
$\kappa$-fully quadratic and $\delta_k\leq \Delta_{\epsilon}\leq 1$,
then if \eqref{eq:tmp:grad}
holds, we have 
\begin{equation}\label{eq:tmp:grad1}
\|g_k\|\ge \|\nabla f(x_k)\|-\kappa_{eg}\d_k\ge
(\zeta-\kappa_{eg})\d_k\ge \max \left\{\eta_2\kappa_{bhm},\kappa_{bhm},
  \frac{8\kappa_{ef}}{\kappa_{scd}(1-\eta_1)}
\right\}\d_k,
\end{equation}
and if  \eqref{eq:tmp:hess} holds, we have
\begin{equation}\label{eq:tmp:hess1}
 -\lambda_{\min}(H_k)\geq -\lambda_{\min}(\nabla^2
 f(X_k))-\kappa_{eh}\delta_k\geq \max\left\{  \eta_2,
     \frac{8\kappa_{ef}}{\kappa_{scd}(1-\eta_1)}\right\}\delta_k.
 \end{equation}
As the estimates $\{ f_k^0,f_k^s \}$ are $\epsilon_F$-s.o.~accurate, with
$\epsilon_F\le \kappa_{ef}$, \eqref{eq:tmp:grad1} and \eqref{eq:tmp:hess1}
imply that condition (\ref{condition.lemma.delta.3_1}) in Lemma
\ref{lemma.delta.3_1} and (\ref{condition.lemma.delta.3_2}) in Lemma
\ref{lemma.delta.3_2} hold, respectively. Thus in both cases,  
 iteration $k$ is successful, i.e. $x_{k+1}=x_k+s_k$ and $\d_{k+1}=\max\{\delta_{max},\gamma\d_k\}$.
 If $I_kJ_k=0$, then $\d_{k+1}\geq\gamma^{-1}\d_k$ simply by 
 the dynamics of Algorithm  \ref{algo:stodfosecond}. 
Finally,  observing that $P\{I_kJ_k\}\geq p=\alpha\beta$ we conclude that  \eqref{Deltakdynamics} implies Assumption \ref{ass:stoch-proc}(ii). 
\end{proof}


To show that Assumption \ref{ass:stoch-proc}(iii) holds, we need an
additional assumption on the accuracy of the function
estimates. We also require, for simplicity, an upper bound on the trust-region radius in Algorithm \ref{algo:stodfosecond}\footnote{This restriction can be avoided if one allows a more involved discussion on dominating
terms in the proofs of Lemmas \ref{lemma.delta.1_1}--\ref{lemma.delta.3_1} and \ref{lemma.delta.4_2}, and in the proof of the main result.}.


\begin{assumption}\label{assumption:expectationF}
\begin{itemize}
\item[(a)]
There exists a constant $\kappa_F$ such that at any iteration $k$,
$$
\E[|F_k^0-f(x_k^0)| |\mathcal{F}_{k-1/2}^{M\cdot F}]\leq \kappa_F \delta^3_k
$$
and
$$
\E[|F_k^s-f(x_k+s_k)| |\mathcal{F}_{k-1/2}^{M\cdot F}]\leq \kappa_F \delta^3_k.
$$
\item[(b)] The upper bound $\delta_{\max}$ in Algorithm \ref{algo:stodfosecond} is chosen so that $\delta_{\max}\leq 1$.  
\end{itemize}
\end{assumption}
 
 Note that the bound on the expectation of $|F_k^0-f(x_k^0)|$ and $|F_k^s-f(x_k+s_k)|$,  in principle, implies that the estimates are $\beta$-probabilistically $\epsilon_F$-s.o.~accurate. However, for $\epsilon_F$ to satisfy the conditions in Assumption  \ref{ass:models_estim2} (b) conditions would have to be imposed on $\kappa_F$. Thus, for our purposes here, we choose to  have any finite $\kappa_F>0$ and to impose the bound only on $\epsilon_F$. 

Assumption \ref{assumption:expectationF}(a) is needed for the case when we have a bad model and bad estimates,
and when the (true) objective may increase after a successful step. Without this assumption, it is possible that the increase
in the objective is at most of order $\delta_k^2$ (due to first order terms), while the decrease (on other successful steps)
may be smaller, of order $\delta_k^3$ (due to second order terms). Such a situation would make it impossible to balance out the
increase and decrease in the objective over the course of the algorithm in such a way to ensure that the stochastic process $\Phi_k$ decreases on average. 

We now prove that Assumption \ref{ass:stoch-proc}(iii) holds for  Algorithm \ref{algo:stodfosecond}.

\begin{theorem}\label{lem:expdec2ndorder}
Let Assumptions \ref{assumptions:Lip2_cont},  \ref{ass:models_estim2}
and \ref{assumption:expectationF} hold. Then,
there exist probabilities $\alpha$ and $\beta$ and a constant $\Theta>0$ such that, conditioned on 
$T_{\epsilon}>k$,  for each iteration $k$ of Algorithm \ref{algo:stodfosecond}, we have 

\begin{equation}
\label{eq:so_dynamics}
1(T_\epsilon > k)\E[\Phi_{k+1}-\Phi_k|\mathcal{F}_{k-1}^{M\cdot F} ]  \leq  -1(T_\epsilon > k)\Theta  \Delta^3_k,
\end{equation}
where $T_{\epsilon}$ is defined in \eqref{eq:Tdef2}, and $\Phi_k$ in \eqref{eq:Phidef_2}.

Moreover, under the particular choice of constants described on page \pageref{page:constants_2}, 
let $\alpha$ and $\beta$ satisfy
\begin{equation}\label{const:alphabeta_2}
(1-\alpha)(1-\beta)\leq \min\left\{0.05, 0.0003\frac{\eta_2\min\{1,\eta_2\}}{\kappa_F}\right\}
\end{equation}
and 
\begin{equation}\label{const:beta_2}
\beta\geq \frac{\kappa_F+0.008\eta_2\min\{1,\eta_2\}}{\kappa_F+0.0085\eta_2\min\{1,\eta_2\}}
\end{equation}
Then,  $\zeta=20\kappa_{bhm}=20(\kappa_{eh}+L)$ and $\Theta\geq  6\cdot 10^{-4}\eta_2\min\{1,\eta_2\}$. 
\end{theorem}

\begin{proof} 
Since \eqref{eq:so_dynamics} easily holds if $T_{\epsilon}\leq k$, we assume in what follows that $T_{\epsilon}>k$ and so
$\tau(x_k)>\epsilon$, where $\tau(x)$ is defined in \eqref{2:f-optimality}. We will consider two possible cases: $\tau(x_k)\geq \zeta \delta_k$ and $\tau(x_k)< \zeta \delta_k$,
where $\zeta$ is defined in \eqref{eq:zeta_2}. We show that \eqref{eq:so_dynamics} holds in both cases and so for all $k<T_{\epsilon}$. Let  $\nu\in (0,1)$ be such that
\begin{eqnarray}\label{eq:nu2}
\frac{\nu}{1-\nu}&\geq & \frac{\gamma^3}{\min \left\{ \zeta C_1, \zeta C_4, C_2 \right\}},
\end{eqnarray}
with $C_1$ defined as in Lemma \ref{lemma.delta.2_1}, $C_4$  in Lemma \ref{lemma.delta.2_2}, and $C_2$ in Lemma \ref{lemma.delta.4_2}.
Note that on all successful iterations, $x_{k+1}=x_k+s_k$ and $\delta_{k+1}=\min\{\gamma \delta_k, \delta_{max}\}$ with $\gamma>1$, hence
\begin{eqnarray} 
\phi_{k+1}-\phi_k\leq\nu (f(x_{k+1})-f(x_k))+(1-\nu)(\gamma^3-1)\delta_k^3.\label{eqn.suc.phi_2}
\end{eqnarray} 
On  all unsuccessful iterations, $x_{k+1}=x_k$ and $\delta_{k+1}=\frac{1}{\gamma} \delta_k$, i.e. 
\begin{eqnarray}
\phi_{k+1}-\phi_k=(1-\nu)(\frac{1}{\gamma^3}-1)\delta_k^3\equiv b_1<0.\label{eqn.unsuc.phi_2}
\end{eqnarray}

{\bf Case 1: $\tau(x_k)=\max \{\|\nabla f(x_k)\|, -\lambda_{\min}(\nabla^2 f(x_k))\}\geq \zeta \delta_k$, where $\zeta$ is defined in \eqref{eq:zeta_2}.}
\begin{itemize}
\item[a.] $I_k=1$ and $J_k=1$, i.e., both the  model and the  estimates are good on iteration $k$.  
From the definition of $\zeta$ and Case 1, we know that  either \eqref{eq:tmp:grad} or \eqref{eq:tmp:hess} hold.
Since $I_k=1$ and $\delta_{\max}\leq 1$ (Assumption \ref{assumption:expectationF}(b)), 
\eqref{eq:tmp:grad} and \eqref{eq:tmp:hess}
imply that either condition (\ref{condition.lemma.delta.2_1}) in Lemma \ref{lemma.delta.2_1} or condition \eqref{condition.lemma.delta.2_2} in Lemma \ref{lemma.delta.2_2}
hold. Therefore, the trial step $s_k$ leads to a decrease in $f$ as in (\ref{eqn.true.decrease_1}) or as in (\ref{eqn.true.decrease_2}). Again from $I_k=1$ and $\delta_{\max}\leq 1$, \eqref{eq:tmp:grad} or \eqref{eq:tmp:hess}
imply that \eqref{eq:tmp:grad1} or \eqref{eq:tmp:hess1} hold, respectively. 
As $J_k=1$,  and $\epsilon_F\le \kappa_{ef}$, \eqref{eq:tmp:grad1} and \eqref{eq:tmp:hess1}
imply that condition (\ref{condition.lemma.delta.3_1}) in Lemma
\ref{lemma.delta.3_1} or (\ref{condition.lemma.delta.3_2}) in Lemma
\ref{lemma.delta.3_2} hold, respectively. Thus in both cases,  
 iteration $k$ is successful, i.e. $x_{k+1}=x_k+s_k$ and $\d_{k+1}=\max\{\delta_{max},\gamma\d_k\}$.

Combining (\ref{eqn.true.decrease_1}) and (\ref{eqn.suc.phi_2}),  we get
\begin{eqnarray}
\phi_{k+1}-\phi_k\le -\nu C_1\| \nabla f(x_k)\|\delta^2_k +(1-\nu)(\gamma^3-1)\delta_k^3,\label{eqn.suc.phi.truedecrease1_2}
\end{eqnarray}
with $C_1$ defined in Lemma \ref{lemma.delta.2_1}.  Since  $\|\nabla f(x_k)  \| \ge \zeta \delta_k$ we have 
\begin{eqnarray}
\phi_{k+1}-\phi_k\le [-\nu C_1 \zeta+(1-\nu)(\gamma^3-1)]\delta_k^3\leq b_1,\label{eq:truedecneg_2}
\end{eqnarray}
with $b_1$ defined in \eqref{eqn.unsuc.phi_2}, for $\nu\in(0,1)$ satisfying \eqref{eq:nu2}.

Combining (\ref{eqn.true.decrease_2}) and (\ref{eqn.suc.phi_2}),  we get
\begin{eqnarray}
\phi_{k+1}-\phi_k\le \nu C_4\lambda_{min}(\nabla^2 f(X^k))\delta^2_k +(1-\nu)(\gamma^3-1)\delta_k^3,\label{eqn.suc.phi.truedecrease1_3}
\end{eqnarray}
with $C_4$ defined in Lemma \ref{lemma.delta.2_2}.  Again, since  $-\lambda_{min}(\nabla^2 f(X^k)) \ge \zeta \delta_k$ we have 
\begin{eqnarray}
\phi_{k+1}-\phi_k \leq [-\nu C_4 \zeta+(1-\nu)(\gamma^3-1)]\delta_k^3\leq b_1,\label{eq:truedecneg_3}
\end{eqnarray}
with $b_1$ defined in \eqref{eqn.unsuc.phi_2}, for $\nu\in(0,1)$ satisfying \eqref{eq:nu2}.

\item[b.] $I_k=1$ and  $J_k=0$, i.e., we have a good model and bad estimates  on iteration  $k$. 
Then the analysis of case (a) applies, and 
either by Lemma \ref{lemma.delta.2_1}  or \ref{lemma.delta.2_2}, $s_k$ yields a sufficient decrease in $f$.  
However, the step can be erroneously  rejected, because of inaccurate probabilistic estimates, 
in which case  we have an unsuccessful iteration and (\ref{eqn.unsuc.phi_2}) holds. 
Since  \eqref{eq:nu2} holds, (\ref{eqn.unsuc.phi_2}) applies  whether the iteration is successful or not.

\item[c.] $I_k=0$  and  $J_k=1$, i.e., we have a bad model and good estimates  on iteration  $k$. 
 In this case, again, iteration $k$ can be either successful or unsuccessful; in the latter case,
 (\ref{eqn.unsuc.phi_2}) holds. In the former, since the estimates are $\epsilon_F$-accurate and \eqref{eq:estimates_2}
 holds,  then by Lemma \ref{lemma.delta.4_2} and Assumption \ref{assumption:expectationF}(b), (\ref{eqn.lemma.4_2}) holds with some 
 $C_2>0$, and so,
\begin{eqnarray}
\phi_{k+1}-\phi_k\le [-\nu C_2+(1-\nu)(\gamma^3-1)]\delta_k^3\leq b_1, \label{2:eqn.suc.phi.truedecrease2}
\end{eqnarray}
due to  the choice of $\nu$  satisfying  \eqref{eq:nu2}.

\item[d.] $I_k=0$ and $J_k=0$, i.e., both the  model and the  estimates are bad on iteration $k$.
Inaccurate estimates can cause the algorithm to accept a bad step, 
which may lead to an increase both in $f$ and in $\delta_k$. Hence in this case $\phi_{k+1}-\phi_{k}$ may  be positive. 
We can derive  a bound on the  increase of $f(x_k)$  on successful steps in terms of the error of the estimates 
\begin{eqnarray}\label{eq:bound_inc_2}
\phi_{k+1}-\phi _{k} &\leq& \nu(f(x_k+s_k)-f(x_k))+(1-\nu) (\gamma^3-1)\delta_k^3 \nonumber  \\&\le& \nu( (f(x_k+s_k)-f_k^s)+ (f_k^s-f_k^0) + (f(x_k)-f_k^0)) +(1-\nu) (\gamma^3-1)\delta_k^3\nonumber  \\ &\leq & \nu( |f(x_k+s_k)-f_k^s|+ |f(x_k)-f_k^0)| +(1-\nu) (\gamma^3-1)\delta_k^3.
 \end{eqnarray}
 On unsuccessful steps (\ref{eqn.unsuc.phi_2}) still applies, which means that the right-hand side of \eqref{eq:bound_inc_2} dominates and \eqref{eq:bound_inc_2}  holds
 whether the iteration is successful or not.  Note that here, in (d), we have not used the definition of Case 1.
 \end{itemize}
\vskip0.5cm

Now we are ready to take the expectation of $\Phi_{k+1}-\Phi_{k}$ in Case 1. 
Case (d) occurs with probability at most $(1-\alpha)(1-\beta)$ and in that case $\phi_{k+1}-\phi_{k}$ is bounded from above as in \eqref{eq:bound_inc_2}. 
Cases (a), (b) and (c) occur otherwise
and in those cases $\phi_{k+1}-\phi_{k}$ is bounded from above by $b_1<0$, with $b_1$ defined in (\ref{eqn.unsuc.phi_2}). 
Hence,  we obtain 
\begin{eqnarray*}
&&\E[\Phi_{k+1}-\Phi_{k}|\mathcal{F}_{k-1}^{M\cdot F}, \{  \tau(X_k)\geq \zeta \Delta_k \} ] \\= &&
\E[\Phi_{k+1}-\Phi_{k}|\mathcal{F}_{k-1}^{M\cdot F}, I_k+J_k=0]+
\E[\Phi_{k+1}-\Phi_{k}|\mathcal{F}_{k-1}^{M\cdot F}, \{  \tau(X_k)\geq \zeta \Delta_k \}, I_k+J_k>0 ] \\ \leq &&
 (1-\alpha)(1-\beta)\left (\nu \E[|f(x_k+s_k)-f_k^s|+ |f(x_k)-f_k^0|| \mathcal{F}_{k-1}^{M\cdot F}]+(1-\nu)(\gamma^3-1)\E[\Delta_k^3|\mathcal{F}_{k-1}^{M\cdot F}]\right )\\+&&
 (1- (1-\alpha)(1-\beta)) (1-\nu)(\frac{1}{\gamma^3}-1)\E[\Delta_k^3|\mathcal{F}_{k-1}^{M\cdot F}]
\end{eqnarray*}
Recalling Assumption \ref{assumption:expectationF},
noting that $\E[\Delta_k^3|\mathcal{F}_{k-1}^{M\cdot F}]=\Delta_k^3$ and rearranging the terms we obtain 
\begin{eqnarray*}
&&\E[\Phi_{k+1}-\Phi_{k}|\mathcal{F}_{k-1}^{M\cdot F}, \{  \tau(X_k)\geq \zeta \Delta_k \}]  \\ \leq &&
 \left ((1-\alpha)(1-\beta)2\nu\kappa_F + (1-\nu)(\frac{1}{\gamma^3}-1)  +  (1-\alpha)(1-\beta) (1-\nu)(\gamma^3-\frac{1}{\gamma^3})\right )\Delta_k^3
\end{eqnarray*}
Choosing $0<\alpha\leq 1$ and  $0<\beta\leq 1$ such that 
\begin{equation}\label{2:alphabeta}
(1-\alpha)(1-\beta)\leq \min\left\{ \frac{1}{2(\gamma^3+1)}, 
\frac{1-\nu}{8\kappa_F\nu}\left(1-\frac{1}{\gamma^3}\right)\right\},
\end{equation}
we conclude that 
\begin{equation}\label{case1:exp}
\E[\Phi_{k+1}-\Phi_{k}|\mathcal{F}_{k-1}^{M\cdot F},\{  \tau(X_k)\geq \zeta \Delta_k \}]  \leq -\frac{1}{4}(1-\nu)\left(1-\frac{1}{\gamma^3}\right)\Delta_k^3.
\end{equation}

{\bf Case 2: $\tau(x_k)=\max \{\|\nabla f(x_k)\|, -\lambda_{\min}(\nabla^2 f(x_k))\}< \zeta \delta_k$, where $\zeta$ is defined in \eqref{eq:zeta_2}.}
\begin{itemize}
\item[i.] $J_k=1$, namely, we have good estimates but the model may be bad. This case follows similarly to 
Case 1(c), and (\ref{2:eqn.suc.phi.truedecrease2}) holds on successful steps. Thus 
decrease $b_1$ can again be guaranteed for $\Phi_k$ whether the iteration is successful or not. 

\item[ii.] $J_k=0$, namely, we have bad estimates and the model may be bad too. In this case, the situation of Case 1(d) may
occur and both $f$ and $\delta_k$ may  increase. Then we can upper bound the potential increase in $\Phi_k$ by
\eqref{eq:bound_inc_2} on both successful and unsuccessful steps\footnote{Note that under additional assumptions on $\kappa_{ef}$
and $\eta_2$, one can further refine the analysis here and take into account the decrease in $\Phi_k$ that could then be achieved
when $I_k=1$}.
\end{itemize}
Now we are ready to take the expectation of $\Phi_{k+1}-\Phi_{k}$ in Case 2. 
Case 2(i) occurs with probability at least $\beta$, 
 and then $\phi_{k+1}-\phi_{k}$ is bounded above by $b_1<0$, with $b_1$ defined in (\ref{eqn.unsuc.phi_2}). 
 Case 2(ii) happens with probability at most $(1-\beta)$, and possible increase in $\Phi_k$ 
 as in \eqref{eq:bound_inc_2}. 
 We obtain
\[
\begin{array}{l}
\E[\Phi_{k+1}-\Phi_{k}|\mathcal{F}_{k-1}^{M\cdot F}, \{  \tau(X_k)< \zeta \Delta_k \} ]  \leq\\[1ex]
(1-\beta)\left (\nu \E[|f(x_k+s_k)-f_k^s|+ |f(x_k)-f_k^0|| \mathcal{F}_{k-1}^{M\cdot F}]+(1-\nu)(\gamma^3-1)\E[\Delta_k^3|\mathcal{F}_{k-1}^{M\cdot F}]\right )\\\\[1ex]+
 \beta (1-\nu)(\frac{1}{\gamma^3}-1)\E[\Delta_k^3|\mathcal{F}_{k-1}^{M\cdot F}]
\end{array}
\]
From Assumption \ref{assumption:expectationF},
and $\E[\Delta_k^3|\mathcal{F}_{k-1}^{M\cdot F}]=\Delta_k^3$  we obtain 
\begin{equation}
\begin{array}{l}
\E[\Phi_{k+1}-\Phi_{k}|\mathcal{F}_{k-1}^{M\cdot F}, \{  \tau(X_k)< \zeta \Delta_k \}]  \leq\\[1ex]
 \left \{(1-\beta)[2\nu\kappa_F + (1-\nu)(\gamma^3-1)]  +  \beta (1-\nu)\left(\frac{1}{\gamma^3}-1\right)\right\}\Delta_k^3
\end{array}
\end{equation}
Choosing $\beta\in (0,1]$ such that
\begin{equation}\label{2:beta-case2}
\frac{\beta}{1-\beta}\geq \frac{2\gamma^3[2\nu\kappa_F + (1-\nu)(\gamma^3-1)]}{(1-\nu)(\gamma^3-1)},
\end{equation}
we conclude that 
\begin{equation}\label{case2:exp}
\E[\Phi_{k+1}-\Phi_{k}|\mathcal{F}_{k-1}^{M\cdot F},\{  \tau(X_k)< \zeta \Delta_k \}]  \leq 
-\frac{1}{2}\beta(1-\nu)\left(1-\frac{1}{\gamma^3}\right)\Delta_k^3.
\end{equation}
Thus, in conclusion, for $\nu$ satisfying \eqref{eq:nu2} and $\alpha$ and $\beta$ satisfying \eqref{2:alphabeta} and \eqref{2:beta-case2}, the expected decrease in $\Phi_k$ in \eqref{eq:so_dynamics} holds, with 
\[
\Theta=\frac{1}{4}\min\{2\beta,1\}(1-\nu)\left(1-\frac{1}{\gamma^3}\right).
\]

Now, let us particularize the constants as on page \pageref{page:constants_2}. Firstly, using $\eta_2\leq 18$
and $\kappa_{bhm}=\kappa_{eh}+L$, we deduce that $\zeta:= 20\kappa_{bhm}=20(\kappa_{eh}+L)$ satisfies \eqref{eq:zeta_2}.
Furthermore, letting $C_1=C_4=\frac{1}{10}$ satisfies the conditions in Lemmas \ref{lemma.delta.2_1} and \ref{lemma.delta.2_2},
and from Lemma \ref{lemma.delta.4_2} and particular choice of $\epsilon_F$ we conclude, 
$C_2=\frac{1}{80}\eta_2\min\{1, \eta_2\}$. Thus, from \eqref{eq:nu2} and $\epsilon_F\leq \kappa_{eh}\leq \kappa_{bhm}$, $\nu$ must satisfy
\[
\frac{\nu}{1-\nu}\geq  \frac{8}{\min \left\{ 2\kappa_{bhm}, C_2 \right\}}=\frac{320}{\eta_2\min\{1, \eta_2\}}.
\]
We let $\nu= \frac{320}{320+ \eta_2\min\{1, \eta_2\}} \in (0,1)$. Then \eqref{2:alphabeta} is equivalent to
\[
(1-\alpha)(1-\beta)\leq \min\left\{\frac{1}{18}, \frac{7\eta_2\min\{1, \eta_2\}}{2^{11}\cdot 10\kappa_F}\right\}
\]
which is implied by \eqref{const:alphabeta_2}. The bound \eqref{2:beta-case2} becomes
\[
\frac{\beta}{1-\beta}\geq 16\left[\frac{640\kappa_F}{7\eta_2\min\{1,\eta_2\}}+1\right], 
\]
which is implied by
\[
\beta \geq \frac{2\cdot 10^3\kappa_F+16\eta_2\min\{1,\eta_2\}}{2\cdot 10^3\kappa_F+17\eta_2\min\{1,\eta_2\}}
\]
which is implied by \eqref{const:beta_2}. The value of $\Theta$ follows also using $\eta_2\min\{1,\eta_2\}\leq 18$.
\end{proof}

Note the difference to first order results (Theorem \ref{thm:main}): the effect of
the stronger assumption on the estimates (Assumption \ref{assumption:expectationF} (a))
 can be clearly seen in our results, in the presence of $\kappa_F$ in the simplified
bounds. Note that, due to the choice of constants and requirements on the accuracy of the estimates,  the $\eta_2$ terms are assumed to be smaller than terms involving Lipschitz constants $\overline{L}$ and $\kappa_{eh}$, and hence, they remain present in the bounds.
Given the definition of $\eta_2$ in the algorithm, we can regard it as a means to control/ensure model quality. 

The main complexity result for second order STORM follows.

\begin{theorem}[Complexity of second order STORM algorithm]\label{the:secondorder}
Consider Algorithm \ref{algo:stodfosecond}
and the corresponding stochastic process.  Let $T_\epsilon$ be defined as in \eqref{eq:Tdef2} with $\epsilon\in (0,1]$. 
 Then, under the assumptions of Theorem \ref{lem:expdec2ndorder}, for sufficiently large $\alpha\in (0,1] $ and $\beta\in (0,1]$
 with $\alpha\beta>1/2$, we have
\begin{equation}\label{Teps_2_final}
\E[T_\epsilon] \leq \frac{\alpha\beta}{2\alpha\beta-1}
 \left(\displaystyle\frac{\Phi_0\zeta^3}
{\Theta\epsilon^3}
+1\right),
\end{equation}
where 
  $\Phi_0$ is defined  in \eqref{eq:Phidef_2} with $k=0$,  $\nu$  in 
\eqref{eq:nu2} and $\zeta$ in \eqref{eq:zeta_2}. Moreover, under the particular choice of constants described on page \pageref{page:constants_2}, and in Theorem \ref{lem:expdec2ndorder}, \eqref{Teps_2_final} becomes
\[
\E[T_\epsilon] \leq 8\cdot 10^3\frac{\alpha\beta}{2\alpha\beta-1}
 \left(\displaystyle\frac{\Phi_0 (\kappa_{eh}+L)^3}{\Theta\epsilon^3}
+1\right),
\]
where $\Theta\geq  6\cdot 10^{-4}\eta_2\min\{1,\eta_2\}$. 
\end{theorem} 

\begin{proof}
The validity of the Assumption \ref{ass:stoch-proc}(iii) follows from  Theorem \ref{lem:expdec2ndorder},
with $h(\delta)= \delta^3$ and $\Delta_{\epsilon}$ defined in \eqref{eq:zeta_2}. Then Lemma \ref{lem:Deltak2} and
the discussion preceding it imply that Theorem \ref{t_eps_bound} applies, which provides \eqref{Teps_2_final}.
\end{proof}

The  $\lim\inf$-type probability one convergence result trivially follows. 
\begin{corollary}[Convergence of second order STORM algorithm]
Under conditions of Theorem \ref{the:secondorder} Algorithm \ref{algo:stodfosecond} generates a subsequence convergent to a second order stationary point, almost surely. 
\end{corollary}

As in Section \ref{1:Storm-examples}, similar set-ups that sub-sample function, gradient and Hessian estimates can be provided 
that satisfy Assumptions \ref{ass:models_estim2} and \ref{assumption:expectationF}
for second order STORM \cite{bandeira2013convergence}.

\section{Conclusion}
We have proposed a general framework based on a stochastic process that can be used to bound expected complexity of optimization algorithms. This framework can be applied beyond the algorithms discussed in this paper and has already been used in a new work on stochastic line search \cite{PaquetteScheinberg2018}. We then applied this framework to  establish that a stochastic trust region method, with dynamic stochastic estimates of the gradient, has essentially the same complexity as any other
first order method in non convex setting. Similarly, given dynamic stochastic estimates of the gradient and Hessian the second order stochastic trust region method converges to second order stationary point and its expected complexity matches the deterministic case. While our algorithm requires the stochastic estimates to be progressively more accurate, it never requires to compute the full gradient, hence it applies in purely stochastic settings.

\bibliographystyle{plain}

\bibliography{stoch_rates}

\section{Appendix}

This Appendix contains proofs of several lemmas that are novel, but
whose proofs are similar to existing results. We include them
here for completeness.

{\bf Proof of Lemma \ref{lemma.delta.1_1}.}\quad
Using the optimal decrease condition \eqref{fod-final}, the upper bound on model Hessian from Lemma \ref{lemma:bmh},  and the fact that $\|g_k\|\ge \kappa_{bhm}\d_k$, we have
\begin{eqnarray*}
m_k(x_k)-m_k(x_k+s_k) \ge\frac{\kappa_{scd}}{2} \| g_k \| \min \left\{ \frac{\| g_k\| }{\| H_k \|} ,\delta_k\right\}= \frac{\kappa_{scd}}{2}\|g_k\|\delta_k.
\end{eqnarray*}

Since the model is $\kappa$-fully quadratic,  the improvement in $f$ achieved by $s_k$ is
\begin{eqnarray*}
&&f(x_k+s_k)-f(x_k)\\
&=& f(x_k+s_k)-m(x_k+s_k)+m(x_k+s_k)-m(x_k)+m(x_k)-f(x_k)\\
&\le &2\kappa_{ef}\d_k^3-\frac{\kappa_{scd}}{2}\|g_k\|\delta_k
\le  - \frac{\kappa_{scd}}{4} \|g_k\|\d_k,
\end{eqnarray*}
where the last inequality is implied by $\d_k^2\leq \d_k\le\frac{\kappa_{scd}}{8\kappa_{ef}} \|g_k\|$.\hfill$\Box$\\

{\bf Proof of Lemma \ref{lemma.delta.2_1}.}\quad
The definition of a $\kappa$-fully-quadratic model yields that
$$\|g_k\|\ge \|\nabla f(x)\|-\kappa_{eg}\delta_k^2 .$$ 
Since condition (\ref{condition.lemma.delta.2_1})  implies that $\| \nabla f(x_k) \| \ge\max\left\{ \kappa_{bhm}+\kappa_{eg},\frac{8\kappa_{ef}}{\kappa_{scd}}+\kappa_{eg}\right\}\d_k   $, using $\d_k\leq 1$, we have
$$\|  g_k \| \ge \max\left\{ \kappa_{bhm}  , \frac{8\kappa_{ef}}{\kappa_{scd}}\right\} \d_k.$$
Hence, the conditions of Lemma \ref{lemma.delta.1_1} hold and we have 
\begin{eqnarray}\label{lemma.delta.2.eqn.1_1}
f(x_k+s_k)-f(x_k)\le  - \frac{\kappa_{scd}}{4} \|g_k\|\d_k .
\end{eqnarray}
Since $\|g_k\|\ge \|\nabla f(x)\|-\kappa_{eg}\delta_k$ in which $\d_k$ satisfies (\ref{condition.lemma.delta.2_1}), we also have 
\begin{eqnarray}\label{lemma.delta.2.eqn.2_1}
\|g_k\| \ge   \max\left\{ \frac{\kappa_{bhm}}{\kappa_{bhm}+\kappa_{eg}},\frac{8\kappa_{ef}}{8\kappa_{ef}+\kappa_{scd}\kappa_{eg}}\right\}\|\nabla f(x_k)\| .
\end{eqnarray}
Combining (\ref{lemma.delta.2.eqn.1_1}) and (\ref{lemma.delta.2.eqn.2_1}) yields (\ref{eqn.true.decrease_1}).\hfill$\Box$\\

{\bf Proof of Lemma \ref{lemma.delta.3_1}.}\quad
Since $  \d_k \le \frac{\| g_k \|}{\kappa_{bhm}}$, the model decrease condition \eqref{fod-final} and the uniform bound on $H_k$ under Lemma  \ref{lemma:bmh} immediately yield that 
\begin{eqnarray}
m_k(x_k) -m_k(x_k+s_k) \ge \dfrac{\kappa_{scd}}{2} \|g_k\| \min\left\{  \dfrac{\| g_k\|}{\kappa_{bhm}  } ,\d_k\right\} = \dfrac{\kappa_{scd}}{2}\| g_k \|\d_k.\label{eqn:deltainc1_1}
\end{eqnarray}
The model $m_k$ being  $(\kappa_{ef},\kappa_{eg},\kappa_{eh})$-fully quadratic implies that
\begin{eqnarray}
|f(x_k)-m_k(x_k)|&\le& \kappa_{ef}\d_{k}^3,\mbox{ and } \label{eqn:deltainc2_1}\\
|f(x_k+s_k)-m_k(x_k+s_k)|&\le& \kappa_{ef}\d_{k}^3.\label{eqn:deltainc3_1}
\end{eqnarray}
Since the estimates are $\epsilon_F$-s.o.accurate with $\epsilon_F\le \kappa_{ef}$, we obtain\begin{eqnarray}
|f_k^0-f(x_k)|\le \kappa_{ef}\d_k^3,\mbox{ and } |f_k^s-f(x_k+s_k)|\le \kappa_{ef}\d_k^3.\label{eqn:deltainc4_1}
\end{eqnarray}

We have
\begin{eqnarray*}
\rho_k &= &\dfrac{f_k^0-f_k^s}{m_k(x_k)-m_k(x_k+s_k)}\\
&=&  \dfrac{f_k^0-f(x_k) }{m_k(x_k)-m_k(x_k+s_k)}+\dfrac{f(x_k) -m_k(x_k)}{m_k(x_k)-m_k(x_k+s_k)}+\dfrac{m_k(x_k)-m_k(x_k+s_k) }{m_k(x_k)-m_k(x_k+s_k)}\\
&& + \dfrac{m_k(x_k+s_k)-f(x_k+s_k)}{m_k(x_k)-m_k(x_k+s_k)}+\dfrac{f(x_k+s_k)   -f_k^s   }{m_k(x_k)-m_k(x_k+s_k)},
\end{eqnarray*}
which, combined with (\ref{eqn:deltainc1_1})-(\ref{eqn:deltainc4_1}), implies
\begin{eqnarray*}
| \rho_k-1|\le \dfrac{8\kappa_{ef}\d_k^3}{\kappa_{scd}\|g_k\| \d_k }\le 1-\eta_1,
\end{eqnarray*}
where we have used the assumptions $\d_k^2\le \d_k \le \frac{ \kappa_{scd}(1-\eta_1) }{ 8\kappa_{ef}  }  \| g_k \|   $ to deduce the last inequality. Hence, $\rho_k\ge\eta_1$. Moreover, since $\|  g_k\|\ge \eta_2\kappa_{bhm}\d_k$, then $\tau_k^m\geq\min\left\{ \| g_k\|, \dfrac{\| g_k\|}{\kappa_{bhm} } \right\} \geq \eta_2\delta_k$ and the $k$-th iteration is successful.
\hfill$\Box$

\end{document}